\documentclass[12pt,a4paper]{article}
\usepackage{amsmath}
\usepackage{amssymb}
\usepackage{amsthm}
\usepackage{float}
\usepackage{tikz}
\usepackage{multirow}
\usepackage{alltt}
\usepackage{bbm}
\usepackage{hyperref}

\tikzset{my loop/.style =  {to path={
  \pgfextra{}
  [looseness=12,min distance=10mm]
  \tikz@to@curve@path},font=\sffamily\small
  }}
\makeatletter

\theoremstyle{definition}
\newtheorem{thm}{Theorem}[section]
\newtheorem{prop}[thm]{Proposition}
\newtheorem{lem}[thm]{Lemma}
\newtheorem{cor}[thm]{Corollary}
\newtheorem{defn}[thm]{Definition}
\newtheorem{conj}[thm]{Conjecture}

\theoremstyle{remark}
\newtheorem*{rem}{Remark}

\DeclareMathOperator{\Hom}{Hom} \DeclareMathOperator{\Aut}{Aut}

\begin{document}

\title{Tensor Product of Polygonal Cell Complexes}
\author{ Yu-Yen Chien \\ Mathematics Division \\ National Center for Theoretical Science \\  \textsc{Taiwan}
\\ \texttt{\small{yychien@ncts.ntu.edu.tw}} }
\date{\today}

\maketitle

\begin{abstract}
We introduce the tensor product of polygonal cell complexes, which
interacts nicely with the tensor product of link graphs of
complexes. We also develop the unique factorization property of
polygonal cell complexes with respect to the tensor product, and
study the symmetries of tensor products of polygonal cell complexes.
\end{abstract}

\section{Introduction}

A {\bf polygonal cell complex} is a 2-dimensional CW-complex with
polygons as 2-cells, namely a graph with polygons attached. To be
precise, we take a rather formal definition: a polygonal cell
complex is a 2-dimensional CW-complex satisfying:

\vspace{-3pt}
\begin{itemize}
\itemsep=-3pt
\item  [(1)] Each 1-cell is an interval of length 1, and each 2-cell
is a disc of positive integral circumference.
\item  [(2)] For a 2-cell of circumference $n$, the attaching map
sends exactly $n$ points evenly distributed on the boundary to the
0-skeleton.
\item  [(3)] For each boundary segment between the points described in (2),
the attaching map sends the segment isometrically onto an open
1-cell.
\end{itemize}
\vspace{-6pt}

Intuitively speaking, we can think of each 2-cell as a regular
polygon, and the attaching map glues vertices to vertices, and edges
to edges. Those 2-cells act like faces of a polyhedron, and we will
use the word face to denote a 2-cell alternatively. Note that the
attaching map of a face may not be injective, and a polygonal cell
complex can be quite different from polyhedra. A {\bf polygonal
complex}, which simulates polyhedra better, is a polygonal cell
complex satisfying:

\vspace{-3pt}
\begin{itemize}
\itemsep=-3pt
\item  [(i)] The attaching map of each cell is injective.
\item  [(ii)] The intersection of any two closed cell is either empty or exactly one closed cell.
\end{itemize}
\vspace{-6pt}

Unless otherwise specified, when we use the word complex, it means
polygonal cell complex, which may or may not be a polygonal complex.

Here is a concise way to describe the local structure of complexes.
For a polygonal cell complex $X$, the {\bf link} of $X$ at a vertex
$v$ is a graph $L(X,v)$ with vertices indexed by ends of edges
attached to $v$, and edges indexed by corners of faces attached to
$v$. Two vertices $v_1$ and $v_2$ in $L(X,v)$ are joined by an edge
$e$ if and only if the corresponding ends of $v_1$ and $v_2$ are
joined by the corresponding corner of $e$. Basically a link
describes the incidence relation of edges and faces at a vertex.
Note that $L(X,v)$ can also be identified as the set $\{x\in X\mid
d(x,v)=\delta\}$, where $d$ is the distance function in $X$ and
$\delta$ is some positive number less than 1/2.

\begin{figure}
\begin{center}
\begin{tikzpicture}[scale=.8]

\filldraw [lightgray](-2,0)--(0,0)--(0,1.1)--cycle;

\draw[fill] (2,0) circle [radius=0.03] (-2,0) circle [radius=0.03]
(0,3.3) circle [radius=0.03]; \draw[->](-1.9,0)--(1.9,0);
\draw[->](-1.95,.0825)--(-0.05,3.2175);
\draw[->](0.05,3.2175)--(1.95,.0825);

\node [left] at (-2,0) {$v$}; \node [right] at (2,0) {$v$}; \node
[above] at (0,3.3) {$v$}; \node [below] at (0,-.05) {$e$}; \node
[left] at (-1.1,1.7) {$e$}; \node [right] at (1.1,1.7) {$e$}; \node
[right] at (0,1.1) {$f$};
\end{tikzpicture}
\caption{a flag of a dunce hat}\label{figa2}
\end{center}
\end{figure}
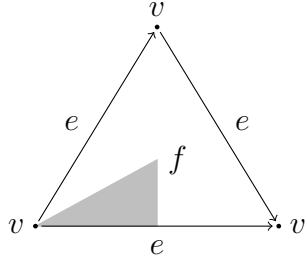

Take the dunce hat in Figure \ref{figa2} as an example. Although
there is only one edge in the complex, this edge has two ends
attached to $v$, and therefore contributes two vertices to the link
at $v$. Notice that the top corner of the face joins these two ends,
and corresponds to an edge joining two vertices in the link at $v$.
The left corner of the face joins the same end of the edge, and
hence corresponds to a loop in the link, while the right corner of
the face also corresponds to a loop at the other vertex. Therefore
the link at $v$ is a graph with two vertices $e_1$ and $e_2$, one
edge joining $e_1$ and $e_2$, and two loops at $e_1$ and $e_2$
respectively.

For polyhedra, a {\bf flag} is an incident triple of face, edge, and
vertex. Such definition needs to be modified for polygonal cell
complexes. Take Figure \ref{figa2} as an example again. It has only
one vertex, one edge, and one face, but we would like it to have six
flags just as a usual triangle. In a polygon, each flag corresponds
to a triangle in its barycentric subdivision. We can use this as an
alternative definition of a flag, and this definition works for
polygonal cell complexes as well. As Figure \ref{figa2} shows, the
shaded area is a flag of the dunce hat, and a dunce hat has six
flags.

Highly symmetric polygonal complexes have been studied in
\cite{bball,jlvv,swiat,valle}. In particular, simply-connected
flag-transitive polygonal complexes with complete graphs as links
are classified in \cite{jlvv}. The main motivation of this paper is
to use these flag-transitive complexes to generate more
flag-transitive complexes. More specifically, we would like to
develop a product of complexes which preserves flag-transitivity,
and the link of the product is some graph product of the links of
factors.

\section{Graph Tensor Product}\label{ch_GTP}

Suppose that $\bullet$ is certain type of graph product such that
$V(\Gamma\bullet\Gamma')=V(\Gamma)\times V(\Gamma')$, and we want to
define a complex product $*$ with the following property: for any
complexes $X$ and $X'$, and for any vertices $v\in X$ and $v'\in
X'$, we have
$$ L(X,v)\bullet L(X',v')\cong L(X*X',(v,v')).$$
Here we have already assumed that $V(X*X')=V(X)\times V(X')$. The
above property provides sufficient information about how the complex
product $*$ shall be defined. If we assume the 1-skeletons of $X$
and $X'$ are simple graphs, by considering the vertex sets of two
link graphs in the equation, we have
$$\{\text{neighbours of }v\text{ in }X\}\times\{\text{neighbours of }v'\text{ in }X'\}=\{\text{neighbours of }(v,v')\text{ in }X*X'\},$$
which can be interpreted as two vertices $(v,v')$ and $(u,u')$ are
adjacent in $X*X'$ if and only if $v$ is adjacent to $u$ in $X$ and
$v'$ is adjacent to $u'$ in $X'$. This is essentially the definition
of the direct product of simple graphs. Since the 1-skeletons of
complexes are not necessarily simple, we shall generalize the direct
product to suit arbitrary graphs.

\begin{figure}
\begin{center}
\tikzstyle{place}=[circle,draw=black,inner sep=0pt,minimum size=2mm]
\begin{tikzpicture}[scale=1]

\node (1) at (0,0) [place] {};

\node (4) at (0,2) [place] {};

\node (7) at (0,-1) [place] {};

\node (x) at (-1.4,0) [place] {};

\node (y) at (-1.4,2) [place] {};

\path (1) ++(2,0) node (2) [place] {} ++(2,0) node (3) [place] {};

\path (4)--++(2,0) node (5) [place] {}--++(2,0) node (6) [place] {};

\path (7)--++(2,0) node (8) [place] {}--++(2,0) node (9) [place] {};

\draw (x) to (y) (7) to (8) (4) to (2) (5) to (1);

\draw (5) to  [out=-45+18,in=135-18] (3);

\draw (5) to  [out=-45-18,in=135+18] (3);

\draw (2) to  [out=45+18, in=225-18](6);

\draw (2) to  [out=45-18, in=225+18](6);

\draw [->>](4) [out=245,in=115] to (1);

\draw [<<-](4) [out=295,in=65] to (1);

\draw [->>](7) to [min distance=12mm, out=135, in=-135] (7);

\draw (8) to [out=25,in=155](9) ;

\draw (8) to [out=-25,in=-155](9) ;
\end{tikzpicture}\vspace{-6mm}
\caption{tensor product of non-simple graphs}\label{fige1}
\end{center}
\end{figure}
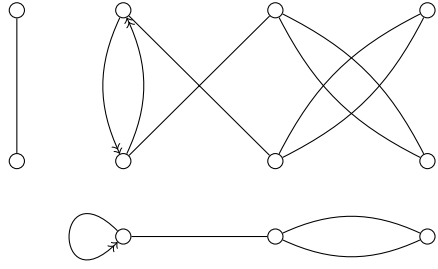

\begin{defn}\label{e1}
Suppose that $\Gamma$ and $\Gamma'$ are two arbitrary graphs with
edge sets $E(\Gamma)=\{e_\alpha\mid\alpha\in A\}$ and
$E(\Gamma')=\{e_\beta\mid\beta\in B\}$. The {\bf tensor product} of
$\Gamma$ and $\Gamma'$, denoted by $\Gamma\otimes\Gamma'$, is a
graph with vertex set $V(\Gamma\otimes\Gamma')=V(\Gamma)\times
V(\Gamma')$, and edge set
$$E(\Gamma\otimes\Gamma')=\big\{ e_{\alpha,\beta}^\delta\mid\alpha\in A, \beta\in B, \delta\in\{0,1\}
\big\},$$ where $e_{\alpha,\beta}^\delta$ is an edge joining
$(v_0,v'_\delta)$ and $(v_1,v'_{1-\delta})$, given $e_\alpha$ joins
$v_0$ and $v_1$ in $\Gamma$, and $e_\beta$ joins $v'_0$ and $v'_1$
in $\Gamma'$.
\end{defn}

Note that for simple graphs, the tensor product defined above is
exactly the direct product of graphs. Like direct product, each pair
of edges from two factors generates two edges in the tensor product,
even when loops are involved, as illustrated in Figures \ref{fige1}
and \ref{fige2}. In some literatures such as \cite{product}, direct
product is defined over graphs without parallel edges but admitting
loops. In such definition, a loop serves as the identity of direct
product. In particular a loop times an edge is an edge, and a loop
times a loop is again a loop, while in our definition a loop times
an edge is two parallel edges, and a loop times a loop creates two
loops around the same vertex. Since we will need such direct product
later, we take a different name and symbol for our generalized
product.

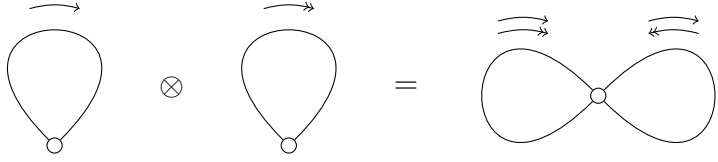
\begin{figure}
\begin{center}
\tikzstyle{place}=[circle,draw=black,inner sep=0pt,minimum size=2mm]
\vspace{-10mm}
\begin{tikzpicture}[scale=1.1]

\node (1) at (0,0) [place] {};

\node (2) at (2.8,0) [place] {};

\node (3) at (6.5,.6) [place] {};

\node at (9,0) {};

\node at (1.4,.7) {$\otimes$};

\node at (4.2,.7) {=};

\draw (1) to [min distance=25mm, out=45, in=135] (1);

\draw (2) to [min distance=25mm, out=45, in=135] (2);

\draw (3) to [min distance=25mm, out=135, in=-135] (3);

\draw (3) to [min distance=25mm, out=45, in=-45] (3);

\draw [->] (-.3,1.65) to [out=15,in=165] (.3,1.65);

\draw [->>] (2.5,1.65) to [out=15,in=165] (3.1,1.65);

\draw [->] (5.3,1.5) to [out=15,in=165] (5.9,1.5);

\draw [->>] (5.3,1.35) to [out=15,in=165] (5.9,1.35);

\draw [->] (7.1,1.5) to [out=15,in=165] (7.7,1.5);

\draw [<<-] (7.1,1.35) to [out=15,in=165] (7.7,1.35);

\end{tikzpicture}\vspace{-10mm}
\caption{tensor product of two loops}\label{fige2}
\end{center}
\end{figure}

There are some reasons to define tensor product in this manner.
First, note the number of vertices in $L(X,v)$ is exactly the
valency of $v$ in $X$, where a loop at $v$ contributes 2 to the
number. Assuming $ L(X,v)\bullet L(X',v')\cong L(X*X',(v,v'))$, this
implies $$d_X(v)\cdot d_{X'}(v')=d_{X*X'}((v,v')),$$ which is true
for the tensor product, but not for the direct product admitting
loops. Secondly, when we glue a face along a loop, the orientation
of gluing matters, and the tensor product can keep track of such
orientations. In Definition \ref{e1}, when $e_\alpha$ or $e_\beta$
is a loop, we shall think of it as an edge joining two different
ends of the loop, say $+$ and $-$, and label two ends of
$e_{\alpha,\beta}^\delta$ by $+$ and $-$ accordingly. We can then
lift any given orientation of a loop in a factor to edges generated
by this loop in the product, as illustrated in Figures \ref{fige1}
and \ref{fige2}. This also allows us to define projections
unambiguously. Note that we do not assume graphs to be directed. We
just distinguish two ends of each loop.

\begin{defn}\label{e2}
Assume the notation of Definition \ref{e1}. The {\bf projection}
from $\Gamma\otimes\Gamma'$ to $\Gamma$, denoted by $\pi_\Gamma$, is
a continuous function such that $\pi_\Gamma$ maps $(v,v')\in
V(\Gamma\otimes\Gamma')$ to $v\in V(\Gamma)$, and
$e_{\alpha,\beta}^\delta\in E(\Gamma\otimes\Gamma')$ to $e_\alpha
\in E(\Gamma)$ isometrically between endpoints. The projection
$\pi_{\Gamma'}$ from $\Gamma\otimes\Gamma'$ to $\Gamma'$ is likewise
defined.
\end{defn}

The projections defined above are graph homomorphisms in the
following sense.

\begin{defn}
Let $\Gamma$ and $\Gamma'$ be two arbitrary graphs. A continuous
function $\varphi$ from $\Gamma$ to $\Gamma'$ is a {\bf
homomorphism} if $\varphi$ maps each vertex of $\Gamma$ to a vertex
of $\Gamma'$, and each open edge of $\Gamma$ isometrically onto an
open edge of $\Gamma'$.
\end{defn}
\begin{rem}
In the above definition, the continuity of $\varphi$ is essentially
saying that a homomorphism maps incident vertices and edges to
incident vertices and edges. Meanwhile, the isometric condition
helps to choose a representative from all homotopic maps.

\end{rem}

Note that the composition of two graph homomorphisms is again a
graph homomorphism. Together with the trivial automorphisms, the
class of graphs forms a category. The following proposition shows
that the tensor product defined above is actually the categorical
product of this category.

\begin{prop}\label{e3}
Let $\Gamma$ and $\Gamma'$ be two arbitrary graphs. Suppose that
$\Gamma_0$ is a graph with two homomorphisms
$\varphi:\Gamma_0\rightarrow\Gamma$ and
$\varphi':\Gamma_0\rightarrow\Gamma'$. Then there exists a unique
homomorphism $\psi:\Gamma_0\rightarrow \Gamma\otimes\Gamma'$ such
that $\varphi=\pi_\Gamma\circ\psi$ and
$\varphi'=\pi_{\Gamma'}\circ\psi$. In other words, there exists a
unique $\psi$ such that the diagram in Figure \ref{fige3} commutes.
\end{prop}
\begin{proof}
Assume that there exists a continuous function
$\psi:\Gamma_0\rightarrow \Gamma\otimes\Gamma'$ such that
$\varphi=\pi_\Gamma\circ\psi$ and $\varphi'=\pi_{\Gamma'}\circ\psi$.
Then $\forall v\in V(\Gamma_0)$, we have
$\varphi(v)=\pi_\Gamma\circ\psi(v)$ and
$\varphi'(v)=\pi_{\Gamma'}\circ\psi(v)$. By Definition \ref{e2}, we
know that $\psi(v)=(\varphi(v),\varphi'(v))$.

Suppose that $e$ is an open edge joining $v$ and $u$ in $\Gamma_0$,
and we denote $\varphi(e)$ and $\varphi'(e)$ by $e_\alpha$ and
$e_\beta$ respectively. By the continuity of $\psi$, $\psi(e)$ is an
open path connecting $(\varphi(v),\varphi'(v))$ and
$(\varphi(u),\varphi'(u))$. Notice that
$e_\alpha=\varphi(e)=\pi_\Gamma\circ\psi(e)$ and
$e_\beta=\varphi'(e)=\pi_{\Gamma'}\circ\psi(e)$. By Definition
\ref{e2}, we know $\psi(e)$ is either $e_{\alpha,\beta}^0$ or
$e_{\alpha,\beta}^1$, determined by endpoints
$(\varphi(v),\varphi'(v))$ and $(\varphi(u),\varphi'(u))$. In case
$e_\alpha$ or $e_\beta$ is a loop, by keeping track of ends of the
loop, $\psi(e)$ is also uniquely determined. Moreover, the local
isometry over open edges of $\varphi$ and $\pi_{\Gamma}$ forces
$\psi$ to map $e$ isometrically to $\psi(e)$. Note that we have
explicitly constructed a continuous $\psi$ satisfying our initial
assumption. We have also shown that $\psi$ is uniquely determined,
and actually a homomorphism, which finishes the proof.
\end{proof}

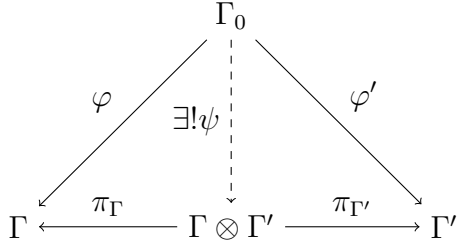
\begin{figure}
\begin{center}
\tikzstyle{place}=[circle,draw=black,inner sep=0pt,minimum size=2mm]
\begin{tikzpicture}[scale=1.4]

\node (1) at (0,0) {$\Gamma\otimes\Gamma'$};

\node (2) at (-2,0) {$\Gamma$};

\node (3) at (2,0) {$\Gamma'$};

\node (4) at (0,2) {$\Gamma_0$};

\draw [->] (1) -- node [above] {$\pi_\Gamma$} (2);

\draw [->] (1) -- node [above] {$\pi_{\Gamma'}$} (3);

\draw [->] (4) -- node [above left] {$\varphi$} (2);

\draw [->] (4) -- node [above right] {$\varphi'$}(3);

\draw [->,dashed] (4) -- node [left] {$\exists!\psi$} (1);

\end{tikzpicture}
\caption{universal property of graph tensor product}\label{fige3}
\end{center}
\end{figure}

For any two graphs $\Gamma$ and $\Gamma'$, we denote the set of all
homomorphisms from $\Gamma$ to $\Gamma'$ by $\Hom(\Gamma,\Gamma'$).
We have the following corollary about the number of homomorphisms.

\begin{cor}\label{e3a}
For any graphs $\Gamma$, $\Gamma_1$, $\Gamma_2$, we have
$$|\Hom(\Gamma,\Gamma_1\otimes\Gamma_2)\,|=|\Hom(\Gamma,\Gamma_1))\,|\cdot|\Hom(\Gamma,\Gamma_2)\,|.$$
\end{cor}
\begin{proof}An immediate consequence of Proposition
\ref{e3}.\end{proof}

Note that for any graph $\Gamma$, there is a homomorphism from
$\Gamma$ to a loop. Since we distinguish the orientations when we
map an edge to a loop, there are actually $2^n$ such homomorphisms,
where $n$ is the number of edges of $\Gamma$. In particular, a loop
is not the terminal object in the category of arbitrary graphs.


\begin{cor}\label{e4}
Let $\Gamma$ and $\Gamma'$ be two graphs, $P$ be a path in $\Gamma$
of length $n$  from $v$ to $u$, and $P'$ be a path in $\Gamma'$ of
length $n$ from $v'$ to $u'$. Then in $\Gamma\otimes\Gamma'$, there
exists a unique path, denoted by $(P,P')_\otimes$, from $(v,v')$ to
$(u,u')$ such that $\pi_\Gamma((P,P')_\otimes)=P$ and
$\pi_{\Gamma'}((P,P')_\otimes)=P'$.
\end{cor}
\begin{proof}
Let $I$ be a graph which is a path of length $n$. We can give $I$ a
specific orientation from one end to the other. Then there is a
natural homomorphism $\varphi$ from $I$ to $P$, as well as one
$\varphi'$ from $I$ to $P'$. By Proposition \ref{e3}, there exists a
unique homomorphism $\psi:I\rightarrow \Gamma\otimes\Gamma'$ such
that $\varphi=\pi_\Gamma\circ\psi$ and
$\varphi'=\pi_{\Gamma'}\circ\psi$. Hence we have
$P=\varphi(I)=\pi_\Gamma\circ\psi(I)$ and
$P'=\varphi'(I)=\pi_{\Gamma'}\circ\psi(I)$. Note that $\psi(I)$
satisfies the conditions of $(P,P')_\otimes$, and the uniqueness of
$(P,P')_\otimes$ follows the uniqueness of $\psi$.
\end{proof}
\begin{rem}
For simple graphs, this result is straightforward from the
definition of tensor product. This corollary clarifies the case when
$P$ or $P'$ contains a loop, where the orientation going through the
loop will determine the edge to choose in $(P,P')_\otimes$.
\end{rem}

\section{Complex Tensor Product}\label{ch_CTP}

To define our complex product more concisely, we would like to
extend the notation $(\;\;,\;\;)_\otimes$ above. Let $\Gamma_1$ and
$\Gamma_2$ be two graphs, $C_1$ be a cycle of length $n$ in
$\Gamma_1$, and $C_2$ be a cycle of length $m$ in $\Gamma_2$. Both
$C_1$ and $C_2$ are assigned initial vertices and orientations.
Specifically, $C_2$ is $(v_0, e_0, v_1, e_1,\ldots, e_{m-1},
v_m=v_0)$, where $v_i\in V(\Gamma_2)$ and $e_j\in E(\Gamma_2)$. Then
for $i\in\{0,1,\ldots,m-1\}$ we define
$$(C_1,C_2)_\otimes^{i^\delta}:=(\frac{[n,m]}{n}C_1,\frac{[n,m]}{m}C_2^{i^\delta})_\otimes,$$
a cycle of length $[n,m]$ in $\Gamma_1\otimes\Gamma_2$, where
$[n,m]$ is the least common multiple of $n$ and $m$, $kC_j$ is the
cycle repeating $C_j$ $k$ times, and $C_2^{i^0}$ is the same cycle
as $C_2$, but starting at $v_i$, while $C_2^{i^1}$ is the reversed
cycle of $C_2$ starting at $v_i$.

\begin{defn}\label{e5}
Let $X$ and $Y$ be two polygonal cell complexes with face sets
$F(X)=\{f_\alpha\mid\alpha\in A\}$ and $F(Y)=\{f_\beta\mid\alpha\in
B\}$. We denote the boundary length of $f_\alpha$ and $f_\beta$ by
$n_\alpha$ and $n_\beta$ respectively, and let $(n_\alpha,n_\beta)$
denote the greatest common divisor of $n_\alpha$ and $n_\beta$. The
{\bf tensor product} of $X$ and $Y$, denoted by $X\otimes Y$, is a
polygonal cell complex with 1-skeleton $X^1\otimes Y^1$, the tensor
product of the 1-skeletons of $X$ and $Y$, and face set
$$F(X\otimes Y)=\big\{f_{\alpha,\beta}^{i^\delta}\mid\alpha\in A, \beta\in B, i\in\{0,1,\ldots,(n_\alpha,n_\beta)-1\}, \delta\in\{0,1\}\big\}, $$
where $f_{\alpha,\beta}^{i^\delta}$ is a face attached along
$(C_\alpha,C_\beta)_\otimes^{i^\delta}$, while $C_\alpha$ is the
cycle along which $f_\alpha$ is attached in $X$, and $C_\beta$ is
the cycle along which $f_\beta$ is attached in $Y$.
\end{defn}

\begin{rem}
We will use the jargon that $f_{\alpha,\beta}^{i^\delta}$ is
generated by $f_\alpha$ and $f_\beta$, especially when faces are not
clearly indexed. In the above definition, note that
$(C_\alpha,C_\beta)_\otimes^{i^\delta}$ and
$(C_\alpha,C_\beta)_\otimes^{{i+(n_\alpha,n_\beta)}^\delta}$ are
identical cycles with different starting vertices. To let a pair of
corners of $f_\alpha$ and $f_\beta$ contribute to exactly one face
corner in $X\otimes Y$, we only choose
$i\in\{0,1,\ldots,(n_\alpha,n_\beta)-1\}$. Here we discard repeated
corner pairs, not faces in $X\otimes Y$ attached along the same
cycle. For example, let $X$ and $Y$ be 15-gons wrapped around a
cycle of length 3 and 5 respectively. Note that the tensor product
of a triangle and a pentagon is not the same as $X\otimes Y$. The
former has only $2\cdot(3,5)=2$ faces, while $X\otimes Y$ has
$2\cdot(15,15)=30$ faces in two groups, each of which has 15 faces
with cyclically identical attaching maps.
\end{rem}

In the example of a triangle tensor a pentagon, the only two faces
meet at every vertex in the product. In general, when $n_\alpha\neq
n_\beta$, two faces $f_{\alpha,\beta}^{i^0}$ and
$f_{\alpha,\beta}^{i^1}$ meet at more than one vertex. Therefore the
tensor product of two polygonal complexes is not necessarily
polygonal. How about the case when $n_\alpha= n_\beta$? For
$n_\alpha$ even, note that $f_{\alpha,\beta}^{i^0}$ and
$f_{\alpha,\beta}^{i^1}$ have two vertices $(0,0)$ and
$(\frac{n_\alpha}{2},\frac{n_\alpha}{2})$ in common, and the tensor
product is not polygonal. For odd cases, we have the following
result.

\begin{prop}\label{e5a}
Suppose that $X$ and $Y$ are polygonal complexes with all faces of
the same odd length $n$. Then the tensor product $X\otimes Y$ is a
polygonal complex.
\end{prop}
\begin{proof}
Since $X$ and $Y$ are polygonal complexes, we know that $X^1$ and
$Y^1$ are simple graphs, and hence the 1-skeleton of $X\otimes Y$,
namely $X^1\otimes Y^1$, is a simple graph as well. Consider the
boundary of an arbitrary face $f_{\alpha,\beta}^{i^\delta}$ in
$X\otimes Y$, namely $(C_\alpha,C_\beta)^{i^\delta}_\otimes$. Note
that $C_\alpha$ and $C_\beta$ are both simple closed cycles of the
same length $n$, as they are boundaries of faces of polygonal
complexes. Therefore $(C_\alpha,C_\beta)^{i^\delta}_\otimes$ is a
simple closed cycle of length $n$. In brief, every face of $X\otimes
Y$ is attached along a simple closed cycle.

Now all we have to show is that the intersection of two faces in
$X\otimes Y$ is either empty, a vertex, or an edge in $X\otimes Y$.
Suppose that there exist two faces $f_{\alpha,\beta}^{i^\delta}$ and
$f_{\alpha',\beta'}^{j^{\delta'}}$ in $X\otimes Y$ such that the
intersection of $f_{\alpha,\beta}^{i^\delta}$ and
$f_{\alpha',\beta'}^{j^{\delta'}}$ is neither empty, a vertex, nor
an edge. For the case of $n=3$, it is not hard to see that
$f_{\alpha,\beta}^{i^\delta}$ and $f_{\alpha',\beta'}^{j^{\delta'}}$
share the same boundary, and in fact are the same face by the
polygonality of $X$ and $Y$. For the case of odd $n>3$, note that
$f_{\alpha,\beta}^{i^\delta}$ and $f_{\alpha',\beta'}^{j^{\delta'}}$
share two vertices which are not consecutive on the boundary of
faces. By the polygonality of $X$ and $Y$, this implies that
$f_\alpha=f_{\alpha'}$ and $f_\beta=f_{\beta'}$. Consider the
boundaries of $f_{\alpha,\beta}^{i^\delta}$ and
$f_{\alpha,\beta}^{j^{\delta'}}$, namely
$(C_\alpha,C_\beta)_\otimes^{i^{\delta}}$ and
$(C_\alpha,C_\beta)_\otimes^{j^{\delta'}}$. When $\delta=\delta'$
and $i\neq j$, $(C_\alpha,C_\beta)_\otimes^{i^{\delta}}$ and
$(C_\alpha,C_\beta)_\otimes^{j^{\delta'}}$ have no vertex in common.
When $\delta\neq\delta'$, notice that a common vertex of
$(C_\alpha,C_\beta)_\otimes^{i^{\delta}}$
 and $(C_\alpha,C_\beta)_\otimes^{j^{\delta'}}$ corresponds to an integer
 $m$ such that $$j+m\equiv i-m\hspace{0cm}\mod
n\hspace{.2cm}\Leftrightarrow\hspace{.2cm} 2m=i-j\mod n,$$ which has
a unique solution when $n$ is odd. In other words, when
$\delta\neq\delta'$, $(C_\alpha,C_\beta)_\otimes^{i^{\delta}}$ and
$(C_\alpha,C_\beta)_\otimes^{j^{\delta'}}$ intersect at exactly one
vertex. Since $f_{\alpha,\beta}^{i^\delta}$ and
$f_{\alpha,\beta}^{j^{\delta'}}$ share two vertices, we can conclude
that $\delta=\delta'$ and $i=j$. This finishes the proof.
\end{proof}

The complex tensor product does not preserve simple connectedness
either.

\begin{prop}\label{e5b}
Let $X$ and $Y$ be an $n$-gon and $m$-gon respectively, where $n$
and $m$ are two positive integers. Then $X\otimes Y$ is
simply-connected if and only if $n=m=1$.
\end{prop}
\begin{proof}
When $n=m=1$, the 1-skeleton of $X\otimes Y$ is a vertex with two
loops, as illustrated in Figure \ref{fige2}, and $X\otimes Y$ has
two faces attached along these two loops respectively. In this case,
$X\otimes Y$ is actually contractible, and of course
simply-connected.

Now suppose that $n$ and $m$ are not both equal to 1. Without loss
of generality, we can assume $n\ge 2$. Note that $X$ has $n$
vertices, $n$ edges, and 1 face, whereas $Y$ has $m$ vertices, $m$
edges, and 1 face. By Definition \ref{e5}, the complex $X\otimes Y$
has $nm$ vertices, $2nm$ edges, and $2(n,m)$ faces. Therefore
$X\otimes Y$ has Euler characteristic
$$\chi(X\otimes Y)=nm-2nm+2(n,m)=-nm+2(n,m)\le-2m+2m=0.$$ By the
following lemma, we know $X\otimes Y$ is not simply-connected.
\end{proof}

\begin{lem}\label{bb}
Suppose $X$ is a finite simply-connected polygonal cell complex.
Then the Euler characteristic of $X$ is at least 1.
\end{lem}
\begin{proof}
Suppose $X$ has $v$ vertices, $e$ edges, and $f$ faces. First we
find an arbitrary spanning tree $T$ for the 1 skeleton of $X$, and
then contract $T$ to get a new complex $X'$, which is also
simply-connected. Note that $T$ has $v-1$ edges, and therefore $X'$
has 1 vertex, $e-v+1$ edges, and $f$ faces. The fundamental group
$\pi_1(X')$, a trivial group, can be presented as a group with
$e-v+1$ generators and $f$ relators. Consider the abelianization of
$\pi_1(X')$, which is again trivial. Then the presentation can be
expressed as $f$ homogeneous equations of $e-v+1$ unknowns over
$\mathbb{Z}$. To have only trivial solution, the number of equations
needs to be at least the number of unknowns. So we have $f\ge
e-v+1$, and therefore $v-e+f\ge 1$.
\end{proof}

\begin{rem}
Let $X$ and $Y$ be two arbitrary complexes, and $C$ be a cycle along
the 1-skeleton of $X\otimes Y$. This proposition shows that the
contractibility of $\pi_{X}(C)$ and $\pi_{Y}(C)$ does not guarantee
the contractibility of $C$. Conversely, when $C$ is contractible in
$X\otimes Y$, can we conclude that $\pi_{X}(C)$ and $\pi_{Y}(C)$ are
contractible? The answer is positive. We can find a series $\{C_j\}$
of homotopic cycles of $C$ such that $C_0=C$, $C_n$ is a vertex, and
each $C_j$ morphs through a single face
$f_{\alpha,\beta}^{i^{\delta}}$ to obtain $C_{j+1}$. Note that
$\pi_X(C_j)$ can morph through a single face $f_\alpha$ to obtain
$\pi_X(C_{j+1})$, even when the length of $f_\alpha$ properly
divides the length of $f_{\alpha,\beta}^{i^{\delta}}$. Therefore
$\pi_X(C)=\pi_X(C_0)$ is homotopic to $\pi_X(C_n)$, which is a
vertex.
\end{rem}

In the above remark, we actually abuse the notation $\pi_{X}$, as we
have not yet defined projection maps for complex tensor products. To
define such projection maps, first we introduce some terminology.
Let $X$ and $Y$ be an $n$-gon and $m$-gon with centre $O_X$ and
$O_Y$ respectively. A function $\rho:X\rightarrow Y$ is {\bf radial}
if $\rho$ sends $O_X$ to $O_Y$, $\partial X$ to $\partial Y$, and
for every point $P\in\partial X$, every real number $t\in [0,1]$, we
have
$$\rho\big(t\cdot O_X+(1-t)P\big)=t\cdot O_Y+(1-t)\rho(P).$$

\begin{defn}\label{e6}
Assume the notation of Definition \ref{e5}. The {\bf projection}
from $X\otimes Y$ to $X$, denoted by $\pi_X$, is a continuous
function such that $\pi_X$ restricted to $X^1\otimes Y^1$ is exactly
$\pi_{X^1}$, the projection of the graph tensor product, and $\pi_X$
maps $f_{\alpha,\beta}^{i^\delta}\in F(X\otimes Y)$ radially to
$f_\alpha\in F(X)$. The projection $\pi_Y$ from $X\otimes Y$ to $Y$
is likewise defined.
\end{defn}

\begin{figure}
\begin{center}
\tikzstyle{place}=[circle,draw=black,inner sep=0pt,minimum size=2mm]
\begin{tikzpicture}[scale=1.4]

\node (0) at (0,0) {};

\path [fill=lightgray]
(0:1)--(60:1)--(120:1)--(180:1)--(240:1)--(300:1)--cycle;

\path [fill=lightgray]
(2.7,-.8)++(60:1.8)--(2.7,-.8)--(4.5,-.8)--cycle;

\path (0) +(0:1) node (1) {$a$} +(60:1) node (2)  {$b$} +(120:1)
node (3) {$c$} +(180:1) node (4) {$a$} +(240:1) node (5) {$b$}
+(300:1) node (6) {$c$} +(1.4,0) node (x) {} +(2.4,0) node (y) {};

\draw (1)--(2)--(3)--(4)--(5)--(6)--(1);

\draw [->] (x)-- node [above] {$\varphi$} (y);

\node (c) at (2.7,-.8) {$c$};

\path (c) +(1.8,0) node (a) {$a$} +(60:1.8) node (b) {$b$};

\draw (a)--(b)--(c)--(a);

\end{tikzpicture}
\caption{a complex homomorphism from a hexagon to a
triangle}\label{fige4}
\end{center}
\end{figure}
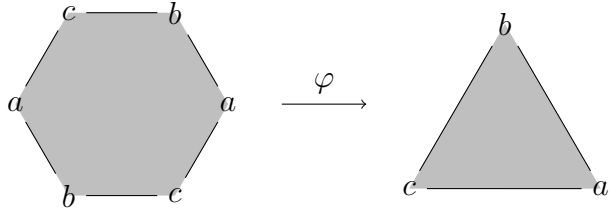

The projection maps defined above are complex homomorphisms in the
following sense.

\begin{defn} \label{e7}
Let $X$ and $Y$ be two polygonal cell complexes. A continuous
function $\varphi$ from $X$ to $Y$ is a {\bf homomorphism} if
$\varphi$ restricted to $X^1$ is a graph homomorphism to $Y^1$, and
$\varphi$ maps each face of $X$ radially to a face of $Y$ and each
open face corner (ignoring the boundary) of $X$ homeomorphically to
an open face corner of $Y$.
\end{defn}
\begin{rem}
In the above definition, the continuity of $\varphi$ is essentially
saying that a complex homomorphism maps incident cells to incident
cells. Similar to the isometric condition in graph homomorphism, the
radial condition is imposed to rule out homotopic complex
homomorphisms. Most important of all, the homeomorphic corner
condition forces a face of $X$ to wrap around a face $f$ of $Y$
along the direction of the attaching map of $f$, possibly more than
once. In particular, a face of length $n$ can only be mapped to a
face of length dividing $n$. Figure \ref{fige4} illustrates such
phenomenon, where corners are mapped to a corner with the same
label. The projection $\pi_X$ of complex tensor product mapping
$f_{\alpha,\beta}^{i^\delta}$ to $f_\alpha$ is also a typical
example.
\end{rem}

Note that the composition of two complex homomorphisms is again a
complex homomorphism. Together with the trivial automorphisms, the
class of polygonal cell complexes forms a category. The following
proposition shows that the complex tensor product defined above is
actually the categorical product of this category.

\begin{figure}
\begin{center}
\tikzstyle{place}=[circle,draw=black,inner sep=0pt,minimum size=2mm]
\begin{tikzpicture}[scale=1.4]

\node (1) at (0,0) {$X\otimes Y$};

\node (2) at (-2,0) {$X$};

\node (3) at (2,0) {$Y$};

\node (4) at (0,2) {$Z$};

\draw [->] (1) -- node [above] {$\pi_X$} (2);

\draw [->] (1) -- node [above] {$\pi_Y$} (3);

\draw [->] (4) -- node [above left] {$\varphi_X$} (2);

\draw [->] (4) -- node [above right] {$\varphi_Y$}(3);

\draw [->,dashed] (4) -- node [left] {$\exists!\psi$} (1);

\end{tikzpicture}
\caption{universal property of complex tensor product}\label{fige5}
\end{center}
\end{figure}
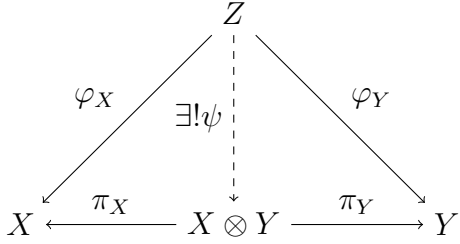

\begin{prop}\label{e7a}
Let $X$ and $Y$ be two polygonal cell complexes. Suppose that $Z$ is
a complex with two homomorphisms $\varphi_X:Z\rightarrow X$ and
$\varphi_Y:Z\rightarrow Y$. Then there exists a unique homomorphism
$\psi:Z\rightarrow X\otimes Y$ such that $\varphi_X=\pi_X\circ\psi$
and $\varphi_Y=\pi_Y\circ\psi$. In other words, there exists a
unique $\psi$ such that the diagram in Figure \ref{fige5} commutes.
\end{prop}
\begin{proof}
Assume that there exists a continuous function $\psi:Z\rightarrow
X\otimes Y$ such that $\varphi_X=\pi_X\circ\psi$ and
$\varphi_Y=\pi_Y\circ\psi$. Note that $\varphi_X$, $\varphi_Y$,
$\pi_X$, and $\pi_Y$ restricted to the 1-skeletons of their domains
are all graph homomorphisms. By Proposition \ref{e3}, the
restriction of $\psi$ to $Z^1$ is a uniquely determined graph
homomorphism to $X^1\otimes Y^1$.

Suppose that $f$ is a face in $Z$, $\varphi_X(f)$ wraps around a
face $f_\alpha$ in $X$, and $\varphi_Y(f)$ wraps around a face
$f_\beta$ in $X$. Then $\varphi_X(f)=\pi_X\circ\psi(f)$ wraps around
$f_\alpha$, and $\varphi_Y(f)=\pi_Y\circ\psi(f)$ wraps around
$f_\beta$. By Definition \ref{e6}, $\psi(f)$ must wrap around
$f_{\alpha,\beta}^{i^\delta}$ for some $i$ and $\delta$. Let $c$ be
a corner of $f$. Then we must have $\varphi_X(c)=\pi_X\circ\psi(c)$
and $\varphi_Y(c)=\pi_Y\circ\psi(c)$. By the remark after Definition
\ref{e5}, this pair of corners $(\varphi_X(c),\varphi_Y(c))$,
orientation included, appears in exactly one
$f_{\alpha,\beta}^{i^\delta}$. Therefore $i$ and $\delta$ are
uniquely determined, and $\psi(f)$ wraps around this
$f_{\alpha,\beta}^{i^\delta}$. Moreover, the radiality of
$\varphi_X$ and $\pi_X$ forces $\psi$ to map $f$ radially to
$f_{\alpha,\beta}^{i^\delta}$. Note that we have explicitly
constructed a continuous $\psi$ satisfying our initial assumption.
We have also shown that $\psi$ is uniquely determined, and actually
a complex homomorphism, which finishes the proof.
\end{proof}
\begin{rem}
For any two complexes $X$ and $Y$, we denote the set of all complex
homomorphisms from $X$ to $Y$ by $\Hom(X,Y)$. Similarly to Corollary
\ref{e3a}, we have $$|\Hom(Z,X\otimes
Y)\,|=|\Hom(Z,X)\,|\cdot|\Hom(Z,Y)\,|.$$
\end{rem}

As we mentioned earlier, for any graph $\Gamma$, there is a
homomorphism from $\Gamma$ to a loop. It is reasonable to ask the
following question: for any complex $X$, is there always a
homomorphism from $X$ to a 1-gon? The answer is negative. Take
Figure \ref{fige6} as an example. Once the image of the leftmost
edge is determined, it determines the image of all other edges. If
we identify the leftmost and the rightmost edges with a twist, i.e.
making it a Mobius strip, then there is no way to have a
homomorphism. Note that this question is not related to
orientability. If the complex is a strip with 3 squares, then the
Mobius case has a homomorphism, while the orientable case does not.

\begin{figure}
\begin{center}
\tikzstyle{place}=[circle,fill=white,draw=black,inner
sep=0pt,minimum size=2mm]
\begin{tikzpicture}[scale=1.8]

\path [fill=lightgray] (0,0)--(2.3,0)--(2.3,1.15)--(0,1.15)--cycle;

\node (x) at (5,0) [place] {};

\draw [<<-] [fill=lightgray] (x) to [min distance=23mm, out=45,
in=135] (x);

\node at (-1,0) {};

\node at (5,0) [place] {};

\node (1) at (0,0) [place] {};

\node (4) at (0,1.15) [place] {};

\path (1) ++(1.15,0) node (2) [place] {} ++(1.15,0) node (3) [place]
{};

\path (4) ++(1.15,0) node (5) [place] {} ++(1.15,0) node (6) [place]
{};

\draw [->>] (1)--(4);

\draw [->>] (4)--(5);

\draw [->>] (5)--(2);

\draw [->>] (2)--(1);

\draw [->>] (6)--(5);

\draw [->>] (3)--(6);

\draw [->>] (2)--(3);

\draw [->] (3.1,.57)--(3.8,.57);

\end{tikzpicture}
\caption{a homomorphism to a 1-gon}\label{fige6}
\end{center}
\end{figure}
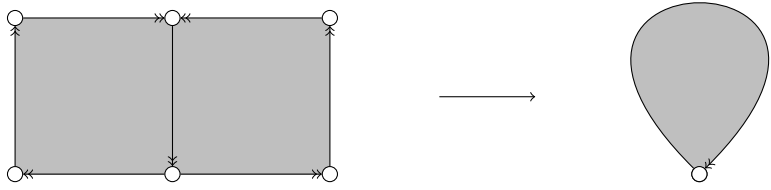

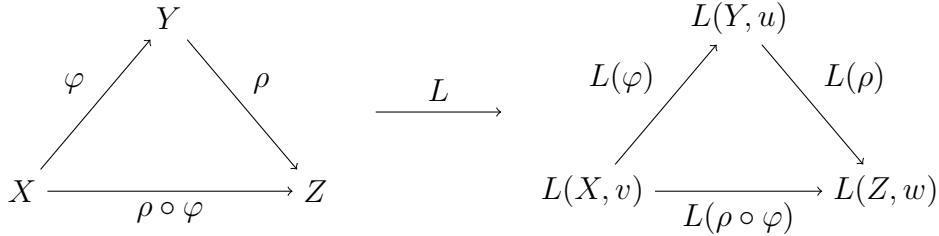
\begin{figure}
\begin{center}
\tikzstyle{place}=[circle,fill=white,draw=black,inner
sep=0pt,minimum size=2mm]
\begin{tikzpicture}[scale=1.5]

\node (x) at (0,0) {$X$};

\node (a) at (5,0) {$L(X,v)$};

\path (x) ++(50:2) node (y) {$Y$} ++(-50:2) node (z) {$Z$};

\draw [->] (x) -- node [above left] {$\varphi$}(y);

\draw [->] (y)-- node [above right] {$\rho$}(z);

\draw [->] (x) -- node [below] {$\rho\circ\varphi$}(z);

\path (a) ++(50:2) node (b) {$L(Y,u)$} ++(-50:2) node (c)
{$L(Z,w)$};

\draw [->] (a) -- node [above left] {$L(\varphi)$}(b);

\draw [->] (b)-- node [above right] {$L(\rho)$}(c);

\draw [->] (a) -- node [below] {$L(\rho\circ\varphi)$}(c);

\draw [->] (3.1,.7)--node[above]{$L$}(4.2,.7);

\end{tikzpicture}
\caption{functoriality of $L$}\label{fige7}
\end{center}
\end{figure}

\begin{prop}
Let $X$ and $Y$ be two polygonal cell complexes, and
$\varphi:X\rightarrow Y$ be a complex homomorphism mapping a vertex
$v\in V(X)$ to $u\in V(Y)$. Then $\varphi$ induces a graph
homomorphism $L(\varphi)$ from $L(X,v)$ to $L(Y,u)$. Moreover, let
$Z$ be another complex and $\rho:Y\rightarrow Z$ be a complex
homomorphism mapping $u$ to $w\in V(Z)$. Then we have
$L(\rho\circ\varphi)=L(\rho)\circ L(\varphi)$, as illustrated in
Figure \ref{fige7}.
\end{prop}
\begin{proof}
By definition, $L(X,v)$ has vertices corresponding to edge ends
around $v$ in $X$, and edges corresponding to face corners at $v$ in
$X$. Since $\varphi$ restricted to $X^1$ is a graph homomorphism,
$\varphi$ maps an edge end around $v$ in $X$ to an edge end around
$u$ in $Y$. In addition, by the homeomorphic condition in Definition
\ref{e7}, $\varphi$ maps a face corner at $v$ joining two edge ends
around $v$ homeomorphically to a face corner at $u$ joining two edge
ends around $u$. Therefore $\varphi$ induces a graph homomorphism
$L(\varphi)$ from $L(X,v)$ to $L(Y,u)$. Once these induced graph
homomorphisms between link graphs are defined, the equality
$L(\rho\circ\varphi)=L(\rho)\circ L(\varphi)$ follows immediately.
\end{proof}
\begin{rem} To each polygonal cell complex, we can assign a
distinguished vertex to be the basepoint. Together with
basepoint-preserving homomorphisms, the class of pointed polygonal
cell complexes also forms a category. The above proposition is
essentially saying that $L$ is a functor from this category to the
category of graphs.
\end{rem}

Now we move back to the main purpose of this chapter: to develop a
complex product interacting nicely with some product of link graphs.
From the above discussion, we know that the complex tensor product
arises naturally in the category of polygonal cell complexes. Does
this natural categorical product fulfill the main job? Yes, it does.

\begin{figure}
\begin{center}
\tikzstyle{place}=[circle,fill=white,draw=black,inner
sep=0pt,minimum size=2mm]
\begin{tikzpicture}[scale=1.9]

\path [fill=lightgray] (0,0)--+(68:.9)--+(112:.9)--cycle;

\path [fill=lightgray] (2.5,0)--+(68:.9)--+(112:.9)--cycle;

\path [fill=lightgray] (5.6,.6)--+(55:.7)--+(125:.7)--cycle;

\path [fill=lightgray] (5.6,.6)--+(-55:.7)--+(-124:.7)--cycle;

\node (v) at (0,0) [place] {};

\path (v) node [below] {$v$};

\draw (v)--node [right] {$e_{\alpha_2}$} +(68:1.5) (v)--node [left]
{$e_{\alpha_1}$} +(180-68:1.5);

\node (u) at (2.5,0) [place] {};

\path (u) node [below] {$u$};

\draw (u)--node [right] {$e_{\beta_2}$} +(68:1.5) (u)--node [left]
{$e_{\beta_1}$} +(180-68:1.5);

\node at (1.25,.6) {$\otimes$};

\node at (3.75,.6) {$=$};

\node (x) at (5.6,.6) [place] {};

\path (x) node [right] {$(v,u)$};

\draw (x)--node [below right] {$(e_{\alpha_2},e_{\beta_2})$}
+(55:1.1) (x)-- node [below left] {$(e_{\alpha_1},e_{\beta_1})$}
+(125:1.1) (x)-- node [above left]
{$(e_{\alpha_1},e_{\beta_2})$}+(235:1.1) (x)--node [above right]
{$(e_{\alpha_2},e_{\beta_1})$}+(305:1.1);

\end{tikzpicture}
\caption{well linked tensor product}\label{fige8}
\end{center}
\end{figure}
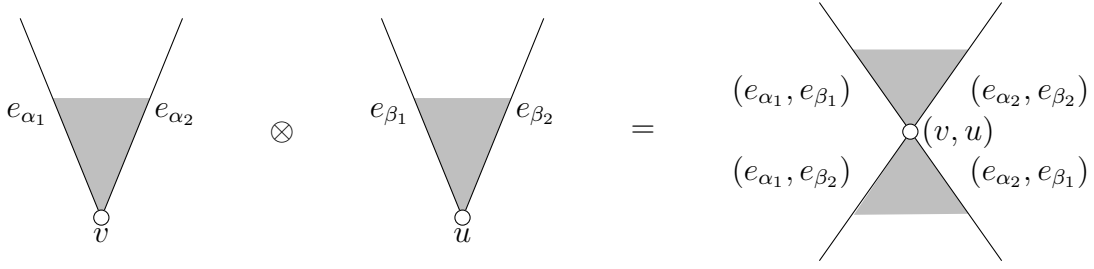

\begin{thm}\label{e8}
Suppose that $X$ and $Y$ are two polygonal cell complexes, and $v$
and $u$ are two vertices in $X$ and $Y$ respectively. Then we have
$$L(X,v)\otimes L(Y,u)\cong L(X\otimes Y,(v,u)).$$
\end{thm}
\begin{proof} We can identify edge ends incident to a vertex as
paths of length 1 leaving the vertex, since a loop contributes to
two edge ends as well as two such paths, which we call 1-paths for
short. By Corollary \ref{e4}, there is a bijection between 1-paths
leaving $(v,u)$ in $X\otimes Y$ and pairs of 1-path leaving $v$ in
$X$ and 1-path leaving $u$ in $Y$. Therefore we can index 1-paths
leaving $(v,u)$ in $X\otimes Y$ by such 1-path pairs in $X$ and $Y$.

Suppose that $f_\alpha\in F(X)$ has a corner $c_\alpha$ at
$(e_{\alpha_1},v,e_{\alpha_2})$, and $f_\beta\in F(Y)$ has a corner
$c_\beta$ at $(e_{\beta_1},u,e_{\beta_2})$, as illustrated in Figure
\ref{fige8}. These $e_*$'s should be understood as 1-paths. By the
remark after Definition \ref{e5}, the pairing of these two corners
appears exactly once in $f_{\alpha,\beta}^{i^0}$ and
$f_{\alpha,\beta}^{j^1}$ respectively, forming corners
$((e_{\alpha_1},e_{\beta_1}),(v,u),(e_{\alpha_2},e_{\beta_2}))$ and
$((e_{\alpha_1},e_{\beta_2}),(v,u),(e_{\alpha_2},e_{\beta_1}))$ in
$X\otimes Y$. Note that by taking projection maps, we know that any
face corner at $(v,u)$ comes from some pairing of corners at $v$ and
$u$.

Now we translate the above statements in terms of corresponding link
graphs. First of all, we have $V\big(L(X,v)\big)\times
V\big(L(Y,u)\big)\cong V\big(L(X\otimes Y,(v,u))\big)$. Secondly,
the corner $c_\alpha$ is an edge joining vertices $e_{\alpha_1}$ and
$e_{\alpha_2}$ in $L(X,v)$, and $c_\beta$ is an edge joining
vertices $e_{\beta_1}$ and $e_{\beta_2}$ in $L(Y,u)$. Notice that
the edge pair $(c_\alpha,c_\beta)$ contributes to one edge joining
$(e_{\alpha_1},e_{\beta_1})$ and $(e_{\alpha_2},e_{\beta_2})$, and
one edge joining $(e_{\alpha_1},e_{\beta_2})$ and
$(e_{\alpha_2},e_{\beta_1})$ in $L(X\otimes Y,(v,u))$. Meanwhile,
taking all possible pairings of edges exhausts all edges in
$L(X\otimes Y,(v,u))$. By Definition \ref{e1}, this is exactly
saying that $L(X,v)\otimes L(Y,u)\cong L(X\otimes Y,(v,u))$.
\end{proof}

\begin{rem}
In the terminology of category theory, this theorem is essentially
saying that the functor $L$ from the category of pointed complexes
to the category of graphs preserves categorical products, which is
not always true for an arbitrary functor.
\end{rem}

As indicated in Propositions \ref{e5a} and \ref{e5b}, the complex
tensor product does not necessarily preserve polygonality and simple
connectedness. Fortunately, complex tensor product does preserve the
most important property for our purpose.

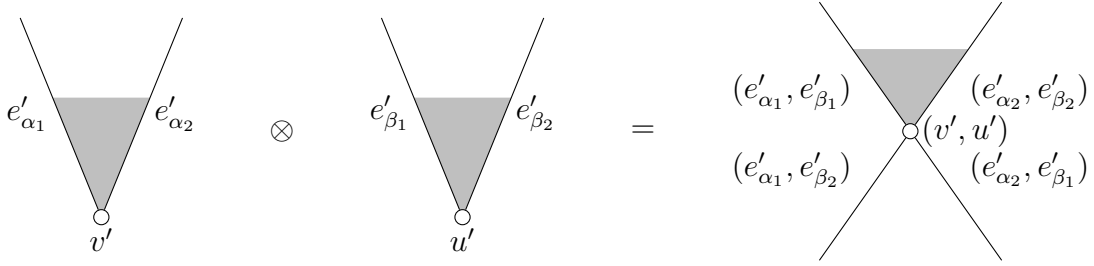
\begin{figure}
\begin{center}
\tikzstyle{place}=[circle,fill=white,draw=black,inner
sep=0pt,minimum size=2mm]
\begin{tikzpicture}[scale=1.9]

\path [fill=lightgray] (0,0)--+(68:.9)--+(112:.9)--cycle;

\path [fill=lightgray] (2.5,0)--+(68:.9)--+(112:.9)--cycle;

\path [fill=lightgray] (5.6,.6)--+(55:.7)--+(125:.7)--cycle;

\node (v) at (0,0) [place] {};

\path (v) node [below] {$v'$};

\draw (v)--node [right] {$e_{\alpha_2}'$} +(68:1.5) (v)--node [left]
{$e_{\alpha_1}'$} +(180-68:1.5);

\node (u) at (2.5,0) [place] {};

\path (u) node [below] {$u'$};

\draw (u)--node [right] {$e_{\beta_2}'$} +(68:1.5) (u)--node [left]
{$e_{\beta_1}'$} +(180-68:1.5);

\node at (1.25,.6) {$\otimes$};

\node at (3.75,.6) {$=$};

\node (x) at (5.6,.6) [place] {};

\path (x) node [right] {$(v',u')$};

\draw (x)--node [below right] {$(e_{\alpha_2}',e_{\beta_2}')$}
+(55:1.1) (x)-- node [below left] {$(e_{\alpha_1}',e_{\beta_1}')$}
+(125:1.1) (x)-- node [above left]
{$(e_{\alpha_1}',e_{\beta_2}')$}+(235:1.1) (x)--node [above right]
{$(e_{\alpha_2}',e_{\beta_1}')$}+(305:1.1);

\end{tikzpicture}
\caption{automorphic image of Figure \ref{fige8}}\label{fige9}
\end{center}
\end{figure}

\begin{thm}\label{e9}
Let $X$ and $Y$ be any two flag-transitive polygonal cell complexes.
Then the complex tensor product $X\otimes Y$ is flag-transitive.
\end{thm}
\begin{proof}
In case $X$ or $Y$ has no faces, then $X\otimes Y$ is simply a
graph, and the flag-transitivity follows easily from the definition
of graph tensor product. Hereafter we assume that both $X$ and $Y$
have at least one face.

Let $((e_{\alpha_1},e_{\beta_1}),(v,u),(e_{\alpha_2},e_{\beta_2}))$
be a face corner in $X\otimes Y$, which projects to a corner
$(e_{\alpha_1},v,e_{\alpha_2})$ in $X$ and a corner
$(e_{\beta_1},u,e_{\beta_2})$ in $Y$, as illustrated in Figure
\ref{fige8}. Let
$((e_{\alpha_1}',e_{\beta_1}'),(v',u'),(e_{\alpha_2}',e_{\beta_2}'))$
be another face corner in $X\otimes Y$, which projects to a corner
$(e_{\alpha_1}',v',e_{\alpha_2}')$ in $X$ and a corner
$(e_{\beta_1}',u',e_{\beta_2}')$ in $Y$, as illustrated in Figure
\ref{fige9}. Since $X$ and $Y$ are flag-transitive, there exist
$\rho\in \Aut(X)$ mapping $(e_{\alpha_1},v,e_{\alpha_2})$ to
$(e_{\alpha_1}',v',e_{\alpha_2}')$ and $\sigma\in\Aut(Y)$ mapping
$(e_{\beta_1},u,e_{\beta_2})$ to $(e_{\beta_1}',u',e_{\beta_2}')$.
Comparing Figures \ref{fige8} and \ref{fige9}, note that
$(\rho,\sigma)$ gives an automorphism of $X\otimes Y$ mapping
$((e_{\alpha_1},e_{\beta_1}),(v,u),(e_{\alpha_2},e_{\beta_2}))$ to
$((e_{\alpha_1}',e_{\beta_1}'),(v',u'),(e_{\alpha_2}',e_{\beta_2}'))$.
The above discussion shows that $\Aut(X\otimes Y)$ acts transitively
on face corners with orientations, and therefore transitively on
half-corners. In other words, $\Aut(X\otimes Y)$ acts transitively
on flags.
\end{proof}
\begin{rem}
In Figure \ref{fige8}, flipping both corners in $X$ and $Y$ will
flip both corners in $X\otimes Y$, whereas flipping only one corner
in either $X$ or $Y$ will swap two corners in $X\otimes Y$.
\end{rem}

\section{Factorization and Symmetry}

In the proof of Theorem \ref{e9}, the key fact we used is the
following relation:
$$\Aut(X)\times\Aut(Y)\le \Aut(X\otimes Y).$$ Is it possible that these two
groups are actually isomorphic? When $X$ an $Y$ are isomorphic, we
can swap $X$ and $Y$ to obtain an extra automorphism, since the
complex tensor product is commutative up to isomorphism. In addition
to swapping, the following proposition gives more extra
automorphisms in a less obvious way.

\vspace{-.7mm}\begin{prop}\label{g1} Let $X$, $Y$, and $Z$ be
polygonal cell complexes. Then we have \vspace{-.7mm}$$(X\otimes
Y)\otimes Z\cong X\otimes(Y\otimes Z).\vspace{-.7mm}$$ In other
words, complex tensor product is associative up to isomorphism.
\end{prop}
\begin{proof} \vspace{-.7mm}A categorical result of the universal
property in Proposition \ref{e7a}. See \cite{maclane}.
\end{proof}

\vspace{-.7mm}The associativity of the complex tensor product
complicates $\Aut(X\otimes Y)$. For example, if $Y$ can be
factorized into $X\otimes Z$, then $X\otimes Y\cong X\otimes
(X\otimes Z)$ has an automorphism swapping the two copies of $X$.
Hence the symmetry of the product of complexes is also related to
the factoring of complexes. In response to associativity, we modify
the original question as follows: for complexes $X_i$ which are
irreducible with respect to complex tensor product, is the
automorphism group $\Aut(\otimes X_i)$ generated by automorphisms of
$X_i$'s, together with permutations of isomorphic factors? By a {\bf
Cartesian automorphism}, we mean an element in the subgroup of
$\Aut(\otimes X_i)$ generated in the above manner.

There have been lots of studies about the symmetry of different
products of graphs. One of the major goals of this chapter is to
apply the theory of the graph direct product to the complex tensor
product. Hence we first introduce related theorems about the graph
direct product. The book \cite{product} by Hammack, Imrich, and
Klav\v{z}ar offers a comprehensive survey of products of graphs, and
we shall follow their approach and terminology here.

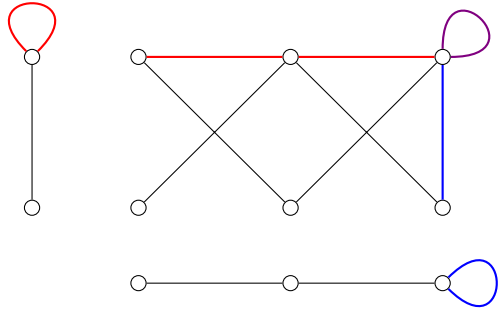
\begin{figure}
\begin{center}
\vspace{-5mm} \tikzstyle{place}=[circle,draw=black,inner
sep=0pt,minimum size=2mm]
\begin{tikzpicture}[scale=1]

\node (1) at (0,0) [place] {};

\node (4) at (0,2) [place] {};

\node (7) at (0,-1) [place] {};

\node (x) at (-1.4,0) [place] {};

\node (y) at (-1.4,2) [place] {};

\path (1) ++(2,0) node (2) [place] {} ++(2,0) node (3) [place] {};

\path (4)--++(2,0) node (5) [place] {}--++(2,0) node (6) [place] {};

\path (7)--++(2,0) node (8) [place] {}--++(2,0) node (9) [place] {};

\draw (x) to (y) (7) to (8) (4) to (2) (5) to (1) (5) to (3) (2) to
(6);

\draw [blue,thick] (6) to (3);

\draw [blue,thick] (9) to [min distance=12mm, out=45, in=-45] (9);

\draw [red,thick] (y) to [min distance=12mm, out=45, in=135] (y);

\draw [red!50!blue,thick] (6) to [min distance=12mm, out=0, in=90]
(6);

\draw [red,thick] (4)--(5)--(6);

\draw (8) to (9) ;

\end{tikzpicture}\vspace{-6mm}
\caption{direct product of graphs in $\mathfrak{S}_0$}\label{figg1}
\end{center}
\end{figure}

We briefly mentioned the direct product of graphs in Chapter
\ref{ch_GTP}. Here we give the definition again, with an emphasis on
the possible presence of loops. We say that a graph $\Gamma$ is a
{\bf simple graph with loops admitted} if for any $u,v\in
V(\Gamma)$, there is at most one edge joining $u$ and $v$, including
the case $u=v$. In particular, there is at most one loop at a
vertex. For convenience, we use $\mathfrak{S}$ to denote the class
of simple graphs, and $\mathfrak{S}_0$ to denote the class of simple
graphs with loops admitted.

\begin{defn}\label{g2}
Let $\Gamma$ and $\Gamma'$ be two graphs in $\mathfrak{S}_0$. The
{\bf direct product} of $\Gamma$ and $\Gamma'$, denoted by
$\Gamma\times\Gamma'$, is a graph in $\mathfrak{S}_0$ with vertex
set $V(\Gamma\times\Gamma')=V(\Gamma)\times V(\Gamma')$. There is an
edge joining two vertices $(v,v')$ and $(u,u')$ in
$\Gamma\times\Gamma'$ if and only if there is an edge joining $v$
and $u$ in $\Gamma$, and there is an edge joining $v'$ and $u'$ in
$\Gamma'$.
\end{defn}

Note\hfill in\hfill the\hfill above\hfill definition,\hfill
$v$\hfill and\hfill $v'$\hfill could\hfill be\hfill the\hfill
same\hfill vertex,\hfill as\hfill well\hfill as\hfill $u$\hfill
and\hfill $u'$.
\\ Figure \ref{figg1} illustrates the direct product of two graphs
in $\mathfrak{S}_0$. Under this definition, notice that a loop $L$
serves as the identity element of direct product of graphs. In other
words, for any simple graph $\Gamma$ with loops admitted, we always
have
$$L\times\Gamma\cong\Gamma\times L\cong\Gamma.$$

 Also note that the
direct product of two edges is again two edges, laid out as a cross
in the figure, which is part of the reason why graph theorists
choose the symbol ``$\times$'' \cite{product}. Therefore the direct
product of two connected graphs is not necessarily connected. The
following theorem  is known as Weichsel's Theorem \cite{product}.

\begin{thm}\label{bf}
Suppose that $\Gamma$ and $\Gamma'$ are two connected simple graphs
with at least two vertices. If $\Gamma$ and $\Gamma'$ are both
bipartite, then $\Gamma\times\Gamma'$ has exactly two components. If
at least one of $\Gamma$ and $\Gamma'$ is not bipartite, then
$\Gamma\times\Gamma'$ is connected.
\end{thm}
\begin{proof}
The first part of the theorem is straightforward. For the second
part, note that a simple graph is not bipartite if and only if there
is an odd cycle in the graph. By exploiting such a cycle properly,
the second part of the theorem follows. For a detailed proof, please
refer to Theorem 5.9 in \cite{product}.
\end{proof}

A graph $\Gamma$ is {\bf prime} if $\Gamma$ has more than one
vertex, and $\Gamma\cong\Gamma_1\times\Gamma_2$ implies that either
$\Gamma_1$ or $\Gamma_2$ is a loop. Note that the idea of being
prime depends on the class of graphs we are talking about. For
example, let $\Gamma$ be a path of length 3, which has 4 vertices.
Then $\Gamma$ is prime in $\mathfrak{S}$, as the only possible
factoring is the product of two edges, which is the disjoint union
of two edges. And the statement that
$\Gamma\cong\Gamma_1\times\Gamma_2$ implies either $\Gamma_1$ or
$\Gamma_2$ is a loop is still logically true. However, $\Gamma$ can
be factorized in $\mathfrak{S}_0$ as the graph on the left of Figure
\ref{figg1} times one edge in the bottom, and hence $\Gamma$ is not
prime in $\mathfrak{S}_0$.

Consider the question of factoring a graph into the product of prime
graphs. For a finite graph, such a prime factorization always
exists, since the number of vertices of factors decreases as the
factoring goes. However, such a prime factorization is not
necessarily unique, and it depends on the graph itself and the class
of graphs where we do the factoring. For example, a path of length 3
together with associativity can be used to create graphs with
non-unique prime factorizations in $\mathfrak{S}$. There are also
graphs with non-unique prime factorizations in $\mathfrak{S}_0$, an
example of which can be found in \cite{product}. The following
theorem of unique prime factorization is due to McKenzie
\cite{mckenzie_ufd}.

\begin{thm}\label{g3}
Suppose that $\Gamma\in\mathfrak{S}_0$ is a finite connected
non-bipartite graph with more than one vertex. Then $\Gamma$ has a
unique factorization into primes in $\mathfrak{S}_0$.
\end{thm}

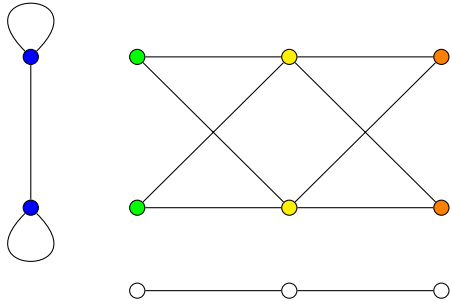
\begin{figure}
\begin{center}
\vspace{-1mm} \tikzstyle{place}=[circle,draw=black,inner
sep=0pt,minimum size=2mm]
\begin{tikzpicture}[scale=1]

\node (1) at (0,0) [place,fill=green] {};

\node (4) at (0,2) [place,fill=green] {};

\node (7) at (0,-1.1) [place] {};

\node (x) at (-1.4,0) [place,fill=blue] {};

\node (y) at (-1.4,2) [place,fill=blue] {};

\node at (4.5,0) {};

\path (1) ++(2,0) node (2) [place,fill=yellow] {} ++(2,0) node (3)
[place,fill=orange] {};

\path (4)--++(2,0) node (5) [place,fill=yellow] {}--++(2,0) node (6)
[place,fill=orange] {};

\path (7)--++(2,0) node (8) [place] {}--++(2,0) node (9) [place] {};

\draw (x) to (y) (7) to (8) (4) to (2) (5) to (1) (5) to (3) (2) to
(6);

\draw (y) to [min distance=12mm, out=45, in=135] (y);

\draw (x) to [min distance=12mm, out=225, in=315] (x);

\draw (1)--(2)--(3) (4)--(5)--(6);

\draw (8) to (9) ;

\end{tikzpicture}\vspace{2mm}
\caption{vertices with the same set of neighbours}\label{figg2}
\end{center}
\end{figure}

The next question is about the automorphism group of direct product,
which hopefully has only these Cartesian automorphisms with respect
to the product. Note that a pair of vertices with the same set of
neighbours creates pairs of vertices with the same set of neighbours
in the direct product, and results in lots of non-Cartesian
automorphisms. This phenomenon is illustrated in Figure \ref{figg2}
, where a vertex with a loop should have itself as a neighbour. We
say that a graph is {\bf $R$-thin} if there are no vertices with the
same set of neighbours. In addition to $R$-thinness, the
disconnectedness due to Theorem \ref{bf} also creates non-Cartesian
automorphisms. Even when the direct product is connected, there
might still be some exotic automorphisms. The following theorem is
due to D\"{o}rfler \cite{dorfler_aut}.

\begin{thm}\label{g4} Suppose that $\Gamma\in\mathfrak{S}_0$ is a finite connected
non-bipartite $R$-thin graph with a prime factorization
$\Gamma=\Gamma_1\times\Gamma_2\times\cdots\times \Gamma_n$ in
$\mathfrak{S}_0$. Then $\Aut(\Gamma)$ is generated by automorphisms
of prime factors and permutations of isomorphic factors.
\end{thm}

We would like to use Theorems \ref{g3} and \ref{g4} to develop
similar results for the complex tensor product. The first problem we
immediately encounter is that, for the complex tensor product, we
obtain the 1-skeleton of the product through the graph tensor
product, which is not exactly the same as the direct product of
graphs. Fortunately, such a difference does not really take place in
graphs with higher symmetries.

\begin{prop}\label{g5}
Let $\Gamma\in\mathfrak{S_0}$ be a finite connected non-bipartite
$R$-thin graph with more than one vertex, and
$\Gamma=\Gamma_1\times\Gamma_2\times\cdots\times \Gamma_n$ be the
unique prime factorization in $\mathfrak{S}_0$. If $\Gamma$ is
edge-transitive, then $\Gamma$ and each prime factor $\Gamma_i$ are
in $\mathfrak{S}$.
\end{prop}
\begin{proof}
Since $\Gamma$ has more than one vertex, the connectedness of
$\Gamma$ implies that $\Gamma$ has a non-loop edge. By the
edge-transitivity of $\Gamma$, we know $\Gamma$ has no loop, and
hence is in $\mathfrak{S}$. If each factor $\Gamma_i$ has a loop,
then the product $\Gamma$ will have a loop, which is not true. If
each factor $\Gamma_i$ is loop-free, then we have finished the
proof. Hence we can assume there is at least one factor with a loop,
and at least one factor without a loop.

Let $\Gamma_\alpha$ be the direct product of all factors with a
loop, and $\Gamma_\beta$ be the direct product of all factors
without a loop. Then we have
$\Gamma=\Gamma_\alpha\times\Gamma_\beta$. Note that permuting
isomorphic factors of $\Gamma$ does not involve permuting factors of
$\Gamma_\alpha$ with factors of $\Gamma_\beta$. By Theorem \ref{g4},
we have $\Aut(\Gamma)=\Aut(\Gamma_\alpha)\times\Aut(\Gamma_\beta)$.
Since a prime factor has more than one vertex, $\Gamma_\alpha$ and
$\Gamma_\beta$ both have more than one vertex. Since $\Gamma$ is
connected, $\Gamma_\alpha$ and $\Gamma_\beta$ are both connected.
Hence $\Gamma_\alpha$ has a loop at some vertex $v$ and a non-loop
edge joining two vertices $v_\alpha$ and $v_\alpha'$, while
$\Gamma_\beta$ has a non loop edge joining two vertices $v_\beta$
and $v_\beta'$. Then in $\Gamma=\Gamma_\alpha\times\Gamma_\beta$,
there is an edge joining $(v,v_\beta)$ and $(v,v_\beta')$, and
another edge joining $(v_\alpha,v_\beta)$ and
$(v_\alpha',v_\beta')$. Notice that
$\Aut(\Gamma)=\Aut(\Gamma_\alpha)\times\Aut(\Gamma_\beta)$ can not
send the first edge to the second one, contradicting the assumption
that $\Gamma$ is edge-transitive.
\end{proof}
\begin{rem}
To visually interpret the last few lines of the proof, it says that
a Cartesian automorphism can not permute horizontal edges with slant
edges in Figure \ref{figg2}.
\end{rem}

Now we move on to the factorization of polygonal cell complexes.
First consider the following example. Let $X$ and $Y$ be a triangle
and a pentagon respectively, $X'$ be a cycle of length 3 with two
triangles attached, and $Y'$ be a cycle of length 5 with two
pentagons attached. Since the numbers of vertices of these complexes
are prime, the only possible way to factorize them is to have a
factor of one vertex with at least a loop and a face, which creates
double edges in the product. Hence we know these complexes can not
be factorized further, and we have non-unique factorizations
$X\otimes Y'\cong X'\otimes Y$.

Here we give another example of non-unique factorization. Let $X$ be
a triangle, and $Y'$ be a $(7\cdot5)$-gon wrapped around a cycle of
length 5. By Definition \ref{e5}, since $3$ and $7\cdot5$ are
coprime, $X\otimes Y'$ has two faces of length $3\cdot 5 \cdot 7$,
wrapped around two cycles of length $3\cdot5$ for $7$ rounds.
Consider a $(7\cdot3)$-gon $X'$ wrapped around a cycle of length 3,
and a pentagon $Y$. It is easy to see that $X\otimes Y'\cong
X'\otimes Y$, and these complexes can not be factorized further. To
avoid these non-uniquely factorized situations, we restrict our
discussion to the factorization of simple complexes.

\begin{defn}\label{g6}
A polygonal cell complex $X$ is a {\bf simple} complex if $X$ has at
least one face, $X$ has no pairs of faces attached along the same
cycle, and the attaching map of each face does not wrap around a
cycle more than once. A polygonal cell complex $X$ is a {\bf prime}
complex if there do not exist complexes $X_1$ and $X_2$ such that
$X=X_1\otimes X_2$.
\end{defn}
\begin{rem}
Figure \ref{figg3} above is a simple complex with two 1-gons. If we
add another 2-gon attached along two different loops, the resulting
complex is still a simple complex, as the boundary cycles of theses
faces are not exactly the same.
\end{rem}

To factorize a complex $X$, our general setting is as follows. We
assume that we know a factorization of the 1-skeleton
$X^1=\Gamma_1\otimes \Gamma_2$, and try to find a complex
factorization $X=X_1\otimes X_2$ such that $X_1^1=\Gamma_1$ and
$X_2^1=\Gamma_2$. A natural thought is to project the faces of $X$
down to $\Gamma_1$ and $\Gamma_2$ to be faces. Consider the complex
tensor product of a triangle and a pentagon, which is a complex with
two 15-gons. Note that when we project these two 15-gons back to the
1-skeletons of factors, what we obtain are 15-gons wrapped around
cycles of length 3 and 5 respectively, not the original faces.


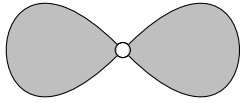
\begin{figure}
\begin{center}
\tikzstyle{place}=[circle,draw=black,inner sep=0pt,minimum size=2mm]
\vspace{-13mm}
\begin{tikzpicture}[scale=1.1]

\node (3) at (6.5,.6) [place] {};

\draw [fill=lightgray] (3) to [min distance=25mm, out=135, in=-135]
(3);

\draw [fill=lightgray] (3) to [min distance=25mm, out=45, in=-45]
(3);

\node at (6.5,.6) [place] {};

\end{tikzpicture}\vspace{-13mm}
\caption{a simple complex with two 1-gons}\label{figg3}
\end{center}
\end{figure}

\begin{defn}\label{g7}
Let $X$ be a polygonal cell complex, $f$ be a face of $X$ attached
along a cycle $C_f$, and $\Gamma_1$ and $\Gamma_2$ be two graphs
such that $X^1=\Gamma_1\otimes \Gamma_2$. The {\bf reductive
projection} of $f$ to $\Gamma_i$, denoted by $\pi_{\Gamma_i}(f)$, is
a face attached along the reduced cycle of $\pi_{\Gamma_i}(C_f)$ in
$\Gamma_i$, namely the shortest cycle $C$ such that repeating $C$
gives $\pi_{\Gamma_i}(C_f)$.
\end{defn}

\begin{rem}
In exactly the same way, we can define $\pi_{\Gamma_i}(f)$ for the
case $X^1=\otimes_{i=1}^n\Gamma_i$. Note that when
$X^1=\Gamma_1\otimes\Gamma_2\otimes\Gamma_3$, we have
$\pi_{\Gamma_1}(f)=\pi_{\Gamma_1}(\pi_{\Gamma_1\otimes\Gamma_3}(f))=\pi_{\Gamma_1}(\pi_{\Gamma_1\otimes\Gamma_2}(f))$.
\end{rem}

\begin{prop}\label{g8}
Let $X$ be a simple complex, and $\Gamma_1$ and $\Gamma_2$ be two
graphs such that $X^1=\Gamma_1\otimes \Gamma_2$. If there exist two
complexes $X_1$ and $X_2$ with 1-skeletons $\Gamma_1$ and $\Gamma_2$
respectively such that $X=X_1\otimes X_2$, then $X_1$ and $X_2$ are
simple complexes whose faces are precisely the reductive projections
of faces of $X$.
\end{prop}
\begin{proof}
Suppose that such complexes $X_1$ and $X_2$ exist. Let $f$ be a face
of $X$ attached along a cycle $C_f$ of length $n$, and let $C_j$ of
length $n_j$ be the reduced cycle of $\pi_{\Gamma_j}(C_f)$ in $X_j$
for $j\in\{1,2\}$. Note that $f$ is generated by a face $f_1$ of
$X_1$ attached along $m_1C_1$, and by a face $f_2$ of $X_2$ attached
along $m_2C_2$, where $m_iC_i$ is the cycle made by repeating $C_i$
for $m_i$ times.  By Definition \ref{e5}, $f_1$ and $f_2$ generate
faces attached along $(m_1C_1,m_2C_2)_\otimes^{i^\delta}$, where
$i\in\{0,1,\ldots,(m_1n_1,m_2n_2)-1\}$ and $\delta\in\{0,1\}$. By
the Euclidean algorithm, we can find an integer $k>0$ such that
$k\equiv 0\mod n_1$ and $k\equiv(n_1,n_2)\mod n_2$. Note that in $k$
steps along $(m_1C_1,m_2C_2)_\otimes^{0^\delta}$, we can walk from
the starting vertex of $(m_1C_1,m_2C_2)_\otimes^{0^\delta}$ to the
starting vertex of $(m_1C_1,m_2C_2)_\otimes^{(n_1,n_2)^\delta}$, so
these two cycles are identical. Since $X$ is simple, there are no
pairs of faces attached along the same cycle in $X$. Therefore we
have $(n_1,n_2)\ge(m_1n_1,m_2n_2)\ge(n_1,n_2).$
Now consider the length of the face $f$, which is
$$n=[m_1n_1,m_2n_2]=\frac{m_1n_1\cdot m_2n_2}{(m_1n_1,m_2n_2)}=\frac{m_1m_2\cdot n_1n_2}{(n_1,n_2)}=m_1m_2\cdot[n_1,n_2].$$
This shows that $f$ is attached along some cycle
$(C_1,C_2)_\otimes^{i^\delta}$ of length $[n_1,n_2]$ for $m_1m_2$
rounds, and the simplicity of $X$ implies that $m_1=m_2=1$. In other
words, $X_i$ must have the reductive projection $\pi_{\Gamma_i}(f)$
of $f$ as its face. Note that different faces of $X$ might have the
same reductive projection in $X_i$, and we have to discard
duplicated ones. Otherwise duplicated faces in $X_i$ will generate
duplicated faces in $X$, violating the simplicity of $X$.
Conversely, any faces $f_1$ of $X_1$ and $f_2$ of $X_2$ are the
reductive projections of the faces in $X$ they generate. Hence $X_1$
and $X_2$ are the simple complexes with exactly those faces from the
reductive projections of faces of $X$.
\end{proof}

\begin{prop}\label{g9}
Let $X$, $X_1$, and $X_2$ be polygonal cell complexes such that
$X=X_1\otimes X_2$. Then $X$ is a simple complex if and only if
$X_1$ and $X_2$ are simple complexes.
\end{prop}
\begin{proof}
Proposition \ref{g8} takes care of the only if part, and here we
prove the if part. Suppose that $X$ has an $n$-gon $f$ attached
along a cycle for $m$ rounds. Since $X_1$ and $X_2$ are simple, $f$
must be generated by the reductive projections of $f$ to $X_1^1$ and
$X_2^1$, which are of length $l_1$ and $l_2$ respectively. Note that
$l_1$ and $l_2$ both divide $\frac{n}{m}$. Then the two reductive
projections generate faces of length $n=[l_1,l_2]\le \frac{n}{m}$.
Hence we can conclude that $m=1$. If there is another face $f'$ in
$X$ attached along the same cycle with $f$, then $f'$ is also
generated by the reductive projections of $f$. If we can show a face
in $X_1$ and a face in $X_2$ do not generate duplicated faces in
$X$, then this implies $X$ is a simple complex.

Suppose that a face $f_1$ of $X_1$ has vertices
$v_0,v_1,\ldots,v_{p-1},v_0$ in order, and a face $f_2$ of $X_2$ has
vertices $u_0,u_1,\ldots,u_{q-1},u_0$ in order. By the remark after
Definition \ref{e5}, every pair of corners of $f_1$ and $f_2$
appears exactly once in the faces generated by $f_1$ and $f_2$. If
two faces generated by $f_1$ and $f_2$ are attached along the same
cycle in $X$, there must be two pairs of corners of $f_1$ and $f_2$
forming the same corner in $X$. In particular, we can find
$(v_i,u_{i'})=(v_j,u_{j'})$ such that $i\neq j$ or $i'\neq j'$. When
$i\neq j$, we have $v_i=v_j$ and $v_{i+k}=v_{j+k}$ for any integer
$k$ mod $p$. This implies that $f_1$ wraps around a cycle more than
once, violating the simplicity of $X_1$. Similarly $i'\neq j'$
contradicts the simplicity of $X_2$. The contradiction results from
the assumption that two faces generated by $f_1$ and $f_2$ are
attached along the same cycle in $X$. Hence we know that $f_1$ and
$f_2$ does not generate duplicated faces, and the simplicity of $X$
follows.
\end{proof}

\begin{prop}\label{g10}
Let $X$ be a simple complex, and $\Gamma_1$ and $\Gamma_2$ be two
graphs such that $X^1=\Gamma_1\otimes\Gamma_2$. Then the following
two statements are equivalent:

\vspace{-3pt}
\begin{itemize}
\itemsep=-3pt
\item  [(1)] There exist two complexes $X_1$ and $X_2$ such that $X_i^1=\Gamma_i$
and $X=X_1\otimes X_2$.
\item  [(2)] For any faces $f_1$ and $f_2$ of $X$, $X$ contains all faces
generated by $\pi_{\Gamma_1}(f_1)$ and $\pi_{\Gamma_2}(f_2)$.
\end{itemize}
\vspace{-8pt}
\end{prop}
\begin{proof}

Assume (1). By Proposition \ref{g8}, $X_1$ and $X_2$ are the simple
complexes with exactly those reductive projections of $X$ as faces.
For any faces $f_1$ and $f_2$ of $X$, $\pi_{\Gamma_1}(f_1)$ is a
face of $X_1$, and $\pi_{\Gamma_2}(f_2)$ is a face of $X_2$. Since
$X=X_1\otimes X_2$, $X$ contains all faces generated by
$\pi_{\Gamma_1}(f_1)$ and $\pi_{\Gamma_2}(f_2)$. Hence (1) implies
(2).

Assume (2). First we show that a face $f$ of $X$ can be generated by
$\pi_{\Gamma_1}(f)$ and $\pi_{\Gamma_2}(f)$. Let $C_f$, $C_1$, and
$C_2$ be the boundary cycles of $f$, $\pi_{\Gamma_1}(f)$, and
$\pi_{\Gamma_2}(f)$ respectively. By Definition \ref{g7}, we can
assume that $\pi_{\Gamma_j}(C_f)=n_jC_j$ for $j\in\{1,2\}$, namely
repeating $C_j$ for $n_j$ times gives $\pi_{\Gamma_j}(C_f)$. Note
that $f$ is attached along some cycle
$(n_1C_1,n_2C_2)_\otimes^{i^\delta}$, which can be rewritten as
$(n_1,n_2)(\frac{n_1}{(n_1,n_2)}C_1,\frac{n_2}{(n_1,n_2)}C_2)_\otimes^{i^\delta}$.
Since the simple complex $X$ has no face attached around a cycle
more than once, we know that $(n_1,n_2)=1$, and therefore
$$\text{length } C_f=n_1\cdot(\text{length } C_1)=n_2\cdot(\text{length } C_2)=[\text{length } C_1,\text{length } C_2].$$
This shows that $\pi_{\Gamma_1}(f)$ and $\pi_{\Gamma_2}(f)$ can
generate the face $f$. Now let $X_1$ and $X_2$ be the simple
complexes with exactly those faces from the reductive projections of
$X$. By Proposition \ref{g9}, $X_1\otimes X_2$ is a simple complex,
and in particular $X_1\otimes X_2$ has no duplicated faces. By the
assumption of (2), $X$ contains all the faces of $X_1\otimes X_2$.
Conversely, any face $f$ of $X$ is a face of $X_1\otimes X_2$, since
$f$ can be generated by $\pi_{\Gamma_1}(f)$ and $\pi_{\Gamma_2}(f)$.
Then we have $X=X_1\otimes X_2$, and hence (2) implies (1).
\end{proof}

Although we already know the associativity of complex tensor product
through the universal property, it will be helpful to understand how
faces are formed in the product of more than two complexes. First
let us review the product of two complexes. Let $f_\alpha$ be a face
of length $n_\alpha$ attached along a cycle $C_\alpha$ in $X$, and
$f_\beta$ be a face of length $n_\beta$ attached along a cycle
$C_\beta$ in $Y$. By Definition \ref{e5}, $f_\alpha$ and $f_\beta$
generate faces $f_{\alpha,\beta}^{i^\delta}$ of length
$[n_\alpha,n_\beta]$ attached along
$(C_\alpha,C_\beta)_\otimes^{i^\delta}$,
$i\in\{0,1,\ldots,(n_\alpha,n_\beta)-1\}$, $\delta\in\{0,1\}$. To
explain the boundary cycle of $f_{\alpha,\beta}^{i^\delta}$ in plain
language, basically we pick a pair of corners of $f_\alpha$ and
$f_\beta$ to start, and go around $C_\alpha$ and $C_\beta$ in two
coordinates respectively until we return to the starting pair of
corners. Note that the index $i$ is chosen in such a way that each
pair of corners appears exactly once among all faces generated by
$f_\alpha$ and $f_\beta$.

\vspace{2mm}

A good way to visualize this is a slot machine of two reels of
length $[n_\alpha,n_\beta]$, cyclically labeled by the vertices of
$f_\alpha$ and $f_\beta$ respectively. Faces generated by $f_\alpha$
and $f_\beta$ have a one-to-one correspondence with different
combinations of two reels, with flipping allowed for the second
reel. From this aspect, it is easy to see that for face $f_j$ of
length $n_j$ in complex $X_j$, $j\in\{1,2,\ldots,m\}$,
$f_1,f_2,\ldots,f_m$ generate faces in $\otimes_{j=1}^m X_j$ of
length $[n_1,n_2,\ldots,n_m]$ such that each $m$-tuple of corners
appears exactly once among all generated faces. Faces generated by
$f_1,f_2,\ldots,f_m$ have a one-to-one correspondence with different
combinations of $m$ reels of length $[n_1,n_2,\ldots,n_m]$,
cyclically labeled by the vertices of $f_j$ respectively, with
flipping allowed from the second reel on. Figure \ref{figg4}
illustrates how a face is generated by the complex tensor product of
3 faces from such an aspect.

\begin{figure}
\begin{center}
\tikzstyle{place}=[circle,draw=black,inner sep=0pt,minimum size=2mm]
\begin{tikzpicture}[scale=.98]

\node (0) at (0,0) {\Huge $\underline{8}$};

\node (1) at (-.17,1.4) [red] {\Huge\bf \underline{7}};

\node (2) at (0,2.8) {\Huge $\underline{6}$};

\node (3) at (+1.45,0) [red] {\Huge\bf \underline{7}};

\node (4) at (-.17+1.45,1.4) {\Huge $\underline{6}$};

\node (5) at (0+1.45,2.8) {\Huge $\underline{5}$};

\node (6) at (+1.45*2,.1) [rotate=180] {\Huge $\underline{6}$};

\node (7) at (-.17+1.45*2,1.5) [red] [rotate=180] {\Huge\bf
\underline{7}};

\node (8) at (0+1.45*2,2.9) [rotate=180] {\Huge $\underline{8}$};

\draw[double] (-.55,-.62) arc (200:160:6);

\draw[double] (+.9,-.62) arc (200:160:6);

\draw[double] (+2.35,-.62) arc (200:160:6);

\draw[double] (+3.9,-.62) arc (200:160:6);

\end{tikzpicture}
\caption{a face generated by 3 faces in complex tensor
product}\label{figg4}
\end{center}
\end{figure}
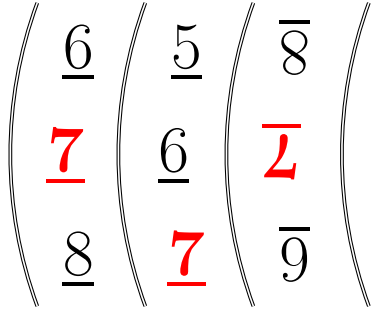

\begin{thm}\label{g11}
Let $X$ be a simple polygonal cell complex. If the 1-skeleton of $X$
is a finite simple connected non-bipartite $R$-thin edge-transitive
graph with more than one vertex, then $X$ has a unique factorization
into prime complexes.
\end{thm}
\begin{proof}
By Theorem \ref{g3}, since $X^1\in\mathfrak{S}\subset\mathfrak{S}_0$
is a finite connected non-bipartite graph with more than one vertex,
$X^1$ has a unique factorization
$X^1=\Gamma_1\times\Gamma_2\times\cdots\times\Gamma_n$ into primes
in $\mathfrak{S}_0$ with respect to direct product of graphs. By
Proposition \ref{g5}, the edge-transitivity of $X^1$ implies that
each prime factor $\Gamma_i$ is in fact a simple graph. On the other
hand, if we factorize $X^1$ with respect to graph tensor product,
each factor would also be a simple graph with more than one vertex,
because a loop creates double edges in the product, and a single
vertex breaks the connectivity of the product. Note that direct
product and graph tensor product coincide in $\mathfrak{S}$. Hence
we know $X^1$ has a unique factorization
$X^1=\Gamma_1\otimes\Gamma_2\otimes\cdots\otimes\Gamma_n$ into
primes in $\mathfrak{S}$ with respect to graph tensor product.

Now we consider the factorization of the complex $X$. Note that we
can always obtain a prime factorization of $X$, since the number of
vertices of factors decreases as the factoring goes. Suppose $X$ has
two factorizations $A$ and $B$, and $X_0$ is a prime factor of $X$
in $A$ with 1-skeleton $\Gamma_1\otimes\Gamma_2$. By Proposition
\ref{g10}, there exist two faces $f_1$ and $f_2$ such that $X_0$
lacks certain face generated by $\pi_{\Gamma_1}(f_1)$ and
$\pi_{\Gamma_2}(f_2)$. In other words, there is certain pair of
corners of $\pi_{\Gamma_1}(f_1)$ and $\pi_{\Gamma_2}(f_2)$ missing
in the faces of $X_0$, and hence such pair will be absent in the
$n$-tuples representing face corners of $X$. By Proposition
\ref{g8}, we can find faces $\overline{f_1}$ and $\overline{f_2}$ of
$X$ such that $\pi_{\Gamma_1\otimes\Gamma_2}(\overline{f_1})=f_1$
and $\pi_{\Gamma_1\otimes\Gamma_2}(\overline{f_2})=f_2$, and we have
$\pi_{\Gamma_1}(\overline{f_1})=\pi_{\Gamma_1}(f_1)$ and
$\pi_{\Gamma_2}(\overline{f_2})=\pi_{\Gamma_2}(f_2)$. If $\Gamma_1$
and $\Gamma_2$ belong to different prime factors $X_1$ and $X_2$ in
$B$, we can reductively project $\overline{f_i}$ to $X_i$
 to obtain a face $f_i'$ of $X_i$, $i\in\{1,2\}$. Then we have
$\pi_{\Gamma_1}(f_1')=\pi_{\Gamma_1}(\overline{f_1})=\pi_{\Gamma_1}(f_1)$
and
$\pi_{\Gamma_2}(f_2')=\pi_{\Gamma_2}(\overline{f_2})=\pi_{\Gamma_2}(f_2)$.
Notice that $f_1'$ and $f_2'$ generate all possible pairs of corners
of $\pi_{\Gamma_1}(f_1)$ and $\pi_{\Gamma_2}(f_2)$ in $X_1\otimes
X_2$ and hence in $X$, a contradiction. So $\Gamma_1$ and $\Gamma_2$
belong to the same prime factor in $B$.

The above argument can be applied to the case when the 1-skeleton of
$X_0$ is the graph tensor product of more than two prime graphs,
simply by splitting prime graph factors into two groups. It follows
that every prime 1-skeleton factor of $X_0$ belongs to the same
prime complex $X_0'$ in $B$. Conversely, every prime 1-skeleton
factor of $X_0'$ belongs to $X_0$, and hence $X_0$ and $X_0'$ are
actually the same. In case $X_0$ has a prime 1-skeleton $\Gamma_j$,
then $\Gamma_j$ belongs to some $X_0'$ in $B$ with a prime
1-skeleton, otherwise the prime 1-skeleton factors of $X_0'$ belong
to at least two complexes in $A$. In conclusion, we know two
factorizations $A$ and $B$ are identical, and $X$ has a unique
factorization into prime complexes.
\end{proof}

\begin{thm}\label{g12}
Suppose that $X$ is a simple polygonal cell complex, and its
1-skeleton is a finite simple connected non-bipartite
edge-transitive $R$-thin graph with more than one vertex. Let
$X=X_1\otimes X_2\otimes\cdots\otimes X_n$ be a prime factorization
of $X$. Then $\Aut(X)$ is generated by automorphisms of prime
factors and permutations of isomorphic factors.
\end{thm}
\begin{proof}
Since $X$ has no faces attached along the same cycle, an
automorphism of $X$ is completely determined by its action on the
1-skeleton $X^1$, and we can identify $\Aut(X)$ as a subgroup of
$\Aut(X^1)$. To understand $\Aut(X^1)$, by the argument in the proof
of Theorem \ref{g11}, we know $X^1$ has a unique factorization
$X^1=\Gamma_1\times\Gamma_2\times\cdots\times\Gamma_m=\Gamma_1\otimes\Gamma_2\otimes\cdots\otimes\Gamma_m$
into primes in $\mathfrak{S}$. By Theorem \ref{g4}, the extra
$R$-thin condition on $X^1$ implies that $\Aut(X^1)$ is generated by
automorphisms of $\Gamma_i$'s and permutations of isomorphic
$\Gamma_j$'s.

Let $\varphi$ be an arbitrary automorphism of $X$, which can be
represented as some $\rho\in\times_{i=1}^m\Aut(\Gamma_i)$ followed
by a permutation of $\Gamma_j$'s. This implies that for any face $f$
of $X$
$$\varphi\big(\pi_{\otimes_{i\in I}\Gamma_i}(f)\big)=\pi_{\varphi(\otimes_{i\in
I}\Gamma_i)}\big(\varphi(f)\big)=\pi_{\otimes_{i\in
I}\varphi(\Gamma_i)}\big(\varphi(f)\big),$$ where $I$ is an
arbitrary non-empty subset of $\{1,2,\ldots,m\}$. Suppose that $X_1$
has 1-skeleton $X_1^1=\otimes_{i\in I}\Gamma_i$ for some
$I\subset\{1,2,\ldots,m\}$. We claim that $\forall i\in I$,
$\varphi(\Gamma_i)$ belongs to the same prime factor $X_k$ of $X$.
If not, then we can find $I_1\sqcup I_2=I$ such that $\forall i\in
I_1,\forall j\in I_2$, $\varphi(\Gamma_i)$ and $\varphi(\Gamma_j)$
belong to different prime factors of $X$. Let
$\Gamma_\alpha=\otimes_{i\in I_1}\Gamma_i$ and
$\Gamma_\beta=\otimes_{j\in I_2}\Gamma_j$, and hence we have
$X_1^1=\Gamma_\alpha\otimes\Gamma_\beta$. Since $X_1$ is prime, by
Proposition \ref{g10}, we can find faces $f_1$ and $f_2$ of $X_1$
such that $X_1$ lacks certain face generated by
$\pi_{\Gamma_\alpha}(f_1)$ and $\pi_{\Gamma_\beta}(f_2)$. By
Proposition \ref{g8}, we can find faces $\overline{f_1}$ and
$\overline{f_2}$ of $X$ such that
$\pi_{\Gamma_\alpha\otimes\Gamma_\beta}(\overline{f_1})=f_1$ and
$\pi_{\Gamma_\alpha\otimes\Gamma_\beta}(\overline{f_2})=f_2$. Then
the complex $X$ lacks certain corner combination of
$\pi_{\Gamma_\alpha}(\overline{f_1})$ and
$\pi_{\Gamma_\beta}(\overline{f_2})$ in the $m$-tuples representing
face corners of $X$. By taking the automorphism $\varphi$, the
complex $X$ lacks certain corner combination of $\pi_{\otimes_{i\in
I_1}\varphi(\Gamma_i)}(\varphi(\overline{f_1}))$ and
$\pi_{\otimes_{j\in
I_2}\varphi(\Gamma_j)}(\varphi(\overline{f_2}))$, which is
impossible because $\varphi(\Gamma_i)$ and $\varphi(\Gamma_j)$
belong to different prime factors of $X$, and taking complex tensor
product of these factors generates all the corner combinations.

Hence for every 1-skeleton factor $\Gamma_i$ of $X_1$,
$\varphi(\Gamma_i)$ belongs to the same prime factor $X_k$ of $X$.
By considering $\varphi^{-1}$, we know that $X_k$ has exactly these
$\varphi(\Gamma_i)$'s as 1-skeleton factors. Moreover,
$\varphi\big(\pi_{\otimes_{i\in
I}\Gamma_i}(f)\big)=\pi_{\otimes_{i\in
I}\varphi(\Gamma_i)}\big(\varphi(f)\big)$ implies that $\varphi$
induces an isomorphism from $X_1$ to $X_k$. This shows that every
$\varphi\in\Aut(X)$ can be represented as some
$\sigma\in\times_{i=1}^n\Aut(X_i)$ followed by a permutation of
$X_j$'s, and the theorem holds.
\end{proof}
\begin{rem}
Let $\widetilde{X}$ be the disjoint union of prime factors of $X$.
Then the above theorem implies that $\Aut(X)\cong
\Aut(\widetilde{X})$, which is a convenient way to describe
$\Aut(X)$.
\end{rem}

The following corollary is a partial converse of Theorem \ref{e9}.

\begin{cor}\label{g13}
Suppose that $X$ is a simple polygonal cell complex, and its
1-skeleton is a finite simple connected non-bipartite
edge-transitive $R$-thin graph with more than one vertex. If $X$ is
flag-transitive, then any factor of $X$ is flag-transitive.
\end{cor}
\begin{proof}

Note that it suffices to show that any prime factor of $X$ is
flag-transitive. Then by Theorem \ref{g11} and Theorem \ref{e9}, any
factor of $X$ is a complex tensor product of flag-transitive prime
factors of $X$, and hence is flag-transitive.

By Theorem \ref{g11}, $X$ has a unique prime factorization
$X=X_1\otimes X_2\otimes\cdots\otimes X_n$. Suppose that one of the
prime factors is not flag-transitive, without loss of generality say
$X_1$, and $X_i$ is isomorphic to $X_1$ if and only of $1\le i\le m$
for some integer $m\le n$. Since $X_1$ is not flag-transitive, there
exist two oriented face corners $(e^1_1,v_1,e^2_1)$ and
$(e^{1'}_1,v'_1,e^{2'}_1)$ in $X_1$ such that $\Aut(X_1)$ can not
map one corner to the other. For each $j$ such that $m+1\le j\le n$,
we pick an arbitrary corner $(e^1_j,v_j,e^2_j)$ of $X_j$. Consider
the following two corners of $X$:
$$((e^1_1,\ldots,e^1_1,e^1_{m+1},\ldots,e^1_n),(v_1,\ldots,v_1,v_{m+1},\ldots,v_n),(e^2_1,\ldots,e^2_1,e^2_{m+1},\ldots,e^2_n))\text{ and}$$
$$((e^{1'}_1,\ldots,e^{1'}_1,e^1_{m+1},\ldots,e^1_n),(v'_1,\ldots,v'_1,v_{m+1},\ldots,v_n),(e^{2'}_1,\ldots,e^{2'}_1,e^2_{m+1},\ldots,e^2_n)).$$
By Theorem \ref{g12}, $\Aut(X)$ is generated by automorphisms of
prime factors and permutation of isomorphic factors. In particular,
it is impossible for $\Aut(X)$ to map one of the above corners to
the other, contradicting to the flag-transitivity of $X$. Therefore
we can conclude that any prime factor of $X$ is flag-transitive.
\end{proof}

The corollary below answers the question we posed in the beginning
of the chapter.

\begin{cor}\label{g14}
For $i\in\{1,2,\ldots,n\}$, let $X_i$ be a simple prime complex with
a finite simple connected non-bipartite symmetric $R$-thin
1-skeleton having more than one vertex. Then the complex tensor
product $X=\otimes_{i=1}^n X_i$ has automorphism group $\Aut(X)$
generated by $\Aut(X_i)$'s and permutations of isomorphic $X_j$'s.
\end{cor}
\begin{proof}
By Proposition \ref{g9}, we know $X$ is a simple complex. By the
definition of graph tensor product, we know $X^1$ is a finite simple
graph. Note that a simple graph is non-bipartite if and only if
there is a cycle of odd length. Then the graph tensor product of two
non-bipartite graphs contains a cycle of odd length and hence is
non-bipartite. Induction shows that $X^1$ is non-bipartite, and by
Theorem \ref{bf} we know that $X^1$ is connected. By the special
case of Theorem \ref{e9} (complexes without faces), we know $X^1$ is
symmetric and hence edge-transitive. Note that for two graphs
$\Gamma_1$ and $\Gamma_2$, the set of neighbours of a vertex
$(u,v)\in V(\Gamma_1\otimes\Gamma_2)$ is the direct product of the
set of neighbours of $u$ in $\Gamma_1$ with the set of neighbours of
$v$ in $\Gamma_2$. This implies the graph tensor product of $R$-thin
graphs is a $R$-thin graph. To summarize, we know $X$ is a simple
complex with a prime factorization $X=\otimes_{i=1}^n X_i$, and its
1-skeleton $X^1$ is a finite simple connected non-bipartite
edge-transitive $R$-thin graph with more than one vertex. By Theorem
\ref{g12}, we know that $\Aut(X)$ is as described in the corollary.
\end{proof}
\begin{rem}
The tensor products of edge-transitive graphs are not necessarily
edge-transitive. Therefore we require each $X_i^1$ to be symmetric
to ensure the edge-transitivity of $X^1$.
\end{rem}

Note that when a complex has a face of odd length, then the
1-skeleton of the complex is non-bipartite, and Corollary \ref{g14}
has a chance to work. In the next chapter, we will investigate the
automorphism group of the tensor product of complexes with only
faces of even lengths from a different aspect.

\section{Even Cases}\label{ch_ts2}

In this chapter we investigate the tensor product of complexes with
only faces of even lengths, and our goal is to develop results
similar to Corollary \ref{g14}, which basically says an automorphism
of certain complex tensor products must be of Cartesian type. Note
that when there is more than one bipartite factor, Theorem \ref{bf}
implies that the complex tensor product is disconnected, and the
product is likely to have non-Cartesian automorphisms from the
direct product of automorphism groups of components. Hence in such a
context, the proper question to pose should be as follows: for
complexes $X_i$ with only faces of even lengths, is the automorphism
group of a component of $\otimes X_i$ generated by automorphisms of
$X_i$'s together with permutations of isomorphic factors?

For graph tensor products, the connectedness of the product does not
guarantee the absence of non-Cartesian automorphisms. For complex
tensor products, we hope that the extra face structure helps to
eliminate non-Cartesian automorphisms. For example, let us look at
the complex tensor product of two squares, which has two isomorphic
components. We denote vertices of a square by $0,1,2,-1$ cyclically,
and illustrate one component of the product in Figure \ref{figh1}.
Note that the 1-skeleton of the component is actually a complete
bipartite graph with $2\cdot4!\cdot4!$ automorphisms, and not all of
them give a complex automorphism due to the extra face structure.

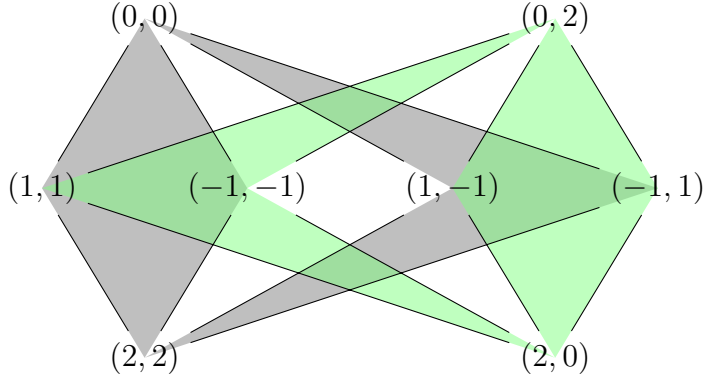
\begin{figure}
\begin{center}
\tikzstyle{place}=[circle,draw=black,inner sep=0pt,minimum size=2mm]
\begin{tikzpicture}[xscale=.9,yscale=.75]

\path [fill=lightgray] (0,0)--(1.5,-3)--(3,0)--(1.5,3)--cycle;

\path [fill=lightgray] (6,0)--(1.5,-3)--(9,0)--(1.5,3)--cycle;

\path [fill=green!50,opacity=.5]
(6,0)--(7.5,-3)--(9,0)--(7.5,3)--cycle;

\path [fill=green!50,opacity=.5]
(0,0)--(7.5,-3)--(3,0)--(7.5,3)--cycle;

\node (1) at (0,0) {($1,1)$};

\node (5) at (1.5,3) {$(0,0)$};

\node (7) at (1.5,-3) {$(2,2)$};

\path (1) ++(3,0) node (2){$(-1,-1)$} ++(3,0) node (3){$(1,-1)$}
++(3,0) node (4){$(-1,1)$};

\path (5)--++(6,0) node (6){$(0,2)$};

\path (7)--++(6,0) node (8){$(2,0)$};

\draw (5)--(1) (5)--(2) (5)--(3) (5)--(4) (6)--(1) (6)--(2) (6)--(3)
(6)--(4) (7)--(1) (7)--(2) (7)--(3) (7)--(4) (8)--(1) (8)--(2)
(8)--(3) (8)--(4);

\end{tikzpicture}
\caption{a component of the tensor product of two
squares}\label{figh1}
\end{center}
\end{figure}

Figure \ref{figh1} also reveals an important fact of the tensor
product of complexes with only faces of even lengths: a face is
antipodally attached to another face generated by the same pair of
faces, and through such antipodally attached relation we can find
all other faces generated by the same pair of faces in that
component. Such face blocks (defined in Definition \ref{h4}) help to
determine the Cartesian structure of a complex tensor product, and
if we can show a generic face block has only Cartesian
automorphisms, then we have a chance to force a complex automorphism
stabilizing a face block to be of Cartesian type. To simplify the
problem, we restrict our discussion to the tensor product of
complexes with faces of the same even length, and the first step is
to establish the Cartesian result for the tensor product of
$2n$-gons. The following lemma is a useful tool for this purpose.

\begin{lem}\label{h1}
Suppose on a real line, someone wants to take $d$ steps to walk from
an integer $d-2k$ to $0$, where $\lfloor\frac{d}{2}\rfloor\ge k\ge
0$ is an integer, and each step is either plus 1 or minus 1. Then
there are ${d-1\choose k}$ ways to arrive from 1, and ${d-1\choose
k-1}$ ways to arrive from $-1$. The ratio ${d-1\choose
k}/{d-1\choose k-1}$ is greater than or equal to 1, with equality if
and only if $d-2k=0$. Moreover, when $d$ is fixed and $k$ is
increasing, the ratio is decreasing.
\end{lem}
\begin{proof}
Suppose this person takes $x$ steps of minus 1 and $y$ steps of plus
1 to arrive at 0. Then we have $x+y=d$ and $-x+y=-d+2k$, and
therefore $x=d-k$ and $y=k$. By ordering two types of steps
arbitrarily, we can obtain all different ways to arrive at 0. To
arrive from $1$, the last step must be minus 1, and there are
${d-1\choose k}$ such combinations. To arrive from $-1$, the last
step must be plus 1, and there are ${d-1\choose k-1}$ such
combinations. When $d$ is odd, we have $\frac{d-1}{2}\ge k$ and
hence ${d-1\choose k}>{d-1\choose k-1}$. When $d$ is even, we have
$\frac{d}{2}\ge k$ which implies $\frac{d-1}{2}>k-1$ and hence
${d-1\choose k}\ge{d-1\choose k-1}$, with equality if and only if
$k+(k-1)=d-1$, namely $d-2k=0$. To show that the ratio ${d-1\choose
k}/{d-1\choose k-1}$ decreases as $k$ increases, we simply have to
verify the following inequality:
\renewcommand{\arraystretch}{1.5}
$$
\begin{array}{lc}
 & {d-1 \choose k}/{d-1 \choose k-1}>{d-1 \choose k+1}/{d-1 \choose k} \\
\Leftrightarrow & {d-1 \choose k}{d-1 \choose k}>{d-1
\choose k+1}{d-1 \choose k-1} \\
\end{array}$$
$$\begin{array}{lc}
\Leftrightarrow & \frac{(d-1)\ldots(d-k)}{k!}\frac{(d-1)\ldots(d-k)}{k!}>\frac{(d-1)\ldots(d-k-1)}{(k+1)!}\frac{(d-1)\ldots(d-k+1)}{(k-1)!}\\
\Leftrightarrow & \frac{d-k}{k}>\frac{d-k-1}{k+1} \\
\Leftrightarrow & \frac{d}{k}-1>\frac{d}{k+1}-1 \\
\Leftrightarrow & k+1>k, \\
\end{array}$$

\noindent which is obviously true.
\end{proof}

\begin{prop}\label{h2}
For $i\in\{1,2,\ldots,m\}$, let $C_i$ be a graph which is a cycle of
length $2n$, where $n$ is an integer at least 3. Then the
automorphism group of a component of $\otimes_{i=1}^m C_i$ can be
generated by elements of $\Aut(C_i)$'s together with permutations of
$C_i$'s.
\end{prop}
\begin{proof}
We denote vertices of $C_i$ by
$0,1,\ldots,n-1,n,-(n-1),-(n-2),\ldots,-1$ cyclically, and let
$\Gamma$ be the component of $\otimes_{i=1}^m C_i$ containing the
vertex $v=(0,0,\ldots,0)$. Note that $\times_{i=1}^m \Aut(C_i)$ acts
transitively on vertices of $\otimes_{i=1}^m C_i$. Therefore to
prove this proposition, it suffices to show that the $v$-stabilizer
$G_v$ of $\Aut(\Gamma)$ can be generated by elements of
$\Aut(C_i)$'s together with permutations of $C_i$'s. Notice that
there are $2^m\cdot m!$ Cartesian automorphisms of $\Gamma$ fixing
$v$, generated by the reflection fixing $0$ in each $C_i$ and all
permutations of $m$ factors. If we can show $|G_v|\le 2^m\cdot m!$,
then the proposition follows.

First we show that $\Gamma$ is a rigid graph. Namely we want to show
that if $\varphi\in G_v$ fixes all neighbours of $v$, then $\varphi$
must be trivial. Note that two vertices $(b_1,b_2,\ldots,b_m)$ and
$(c_1,c_2,\ldots,c_m)$ are adjacent if and only if
$b_i-c_i\equiv\pm1\hspace{0mm}\mod 2n$ for all $i$, and therefore
$$V(\Gamma)\subseteq V^*=\{(a_1,a_2,\ldots,a_m)\in
V(\otimes_{i=1}^m C_i)\mid a_1\equiv a_2\equiv\cdots\equiv
a_m\hspace{-0mm}\mod 2\}.$$ For each $u=(a_1,a_2,\ldots,a_m)\in
V^*$, there is a path of length $d=\max\{|a_1|,|a_2|,\ldots,|a_m|\}$
from $u$ to $v$, because we can reach 0 in $d$ steps in the
coordinates with absolute value $d$, and we can also reach 0 in $d$
steps in the other coordinates by walking back and forth as each
coordinate has the same parity. Hence $V(\Gamma)=V^*$, and
$d(u,v)=d$ follows easily.

Note that the number of geodesics from $u$ to $v$ is the product of
the number of ways in each coordinate to walk to 0 in $d$ steps.
Look at the $i$-th coordinate of $v$. For now we assume that $a_i\ge
0$, and let $k_i$ be the integer such that $a_i=d-2k_i$. If
$n>a_i>0$, we have $\lfloor\frac{d}{2}\rfloor\ge k_i\ge 0$, and
walking to 0 in $d$ steps is equivalent to the setting of Lemma
\ref{h1}. By the lemma, the ratio of numbers of $u-v$ geodesics
arriving from 1 and from $-1$ in the $i$-th coordinate is
${d-1\choose k_i}/{d-1\choose k_i-1}>1$. Since the automorphism
$\varphi$ fixes $(\pm1,\pm1,\ldots,\pm1)$ and preserves geodesics,
this ratio does not change under $\varphi$. Again by the Lemma,
$k_i$ must remain the same to keep this ratio, and hence the $i$-th
coordinate of $\varphi(u)$ must be $a_i$. If $a_i=0$, then $u$ has a
neighbour $w$ with the $i$-th coordinate 1. Note that $\varphi(u)$
is adjacent to $\varphi(w)$ with the $i$-th coordinate 1, and the
$i$-th coordinate of $\varphi(u)$ is either 0 or 2. In the latter
case, since $n>2>0$, by taking $\varphi^{-1}$ the above argument
implies $a_i=2$, a contradiction. Hence the $i$-th coordinate of
$\varphi(u)$ is $0$. Similarly if $a_i=n$, then the $i$-th
coordinate of $\varphi(u)$ is $n$. For negative $a_i$, by applying
the mirror version of Lemma \ref{h1}, we know that the $i$-th
coordinate of $\varphi(u)$ is $a_i$. Note that the above result is
true for every coordinate. Hence $\varphi(u)=u$ for every $u\in
V(\Gamma)$, and $\varphi$ is trivial.

Now look at the local structure around $v$. Note that two neighbours
of $v$ taking different values in $k$ coordinates have $2^{m-k}$
common neighbours. In particular, two neighbours of $v$ differ in
exactly one coordinate if and only of they have $2^{m-1}$ common
neighbours. Hence among the neighbours of $v$, the relation of
differing in exactly one coordinate is preserved under $G_v$. If we
draw an auxiliary edge between any two such neighbours of $v$, then
the $2^m$ neighbours of $v$ plus these auxiliary edges form a
hypercube $Q_m$ preserved under $G_v$. Since $\Gamma$ is rigid, an
automorphism of $G_v$ is completely determined by its action on the
neighbours of $v$, which also induces an automorphism of the
auxiliary $Q_m$. As a result, we have $|G_v|\le\Aut(Q_m)=2!\cdot
m!$, which finishes the proof.
\end{proof}
\begin{rem}
Let $H$ be the subgroup of $\times_{i=1}^m \mathbb{Z}_{2n}$
generated by $S=\{(\pm1,\pm1,\ldots,\pm1)\}$. Note that the
component $\Gamma$ in the above proof is actually isomorphic to the
Cayley graph of $H$ with respect to the generating set $S$.
\end{rem}

\begin{cor}\label{h3}
Suppose that $X_i$ is a $2n$-gon for $i\in\{1,2,\ldots,m\}$, where
$n$ is an integer at least 3. Then the automorphism group of a
component of $\otimes_{i=1}^m X_i$ can be generated by elements of
$\Aut(X_i)$'s together with permutations of $X_i$'s.
\end{cor}
\begin{proof}
Note that a $2n$-gon has the same automorphism group as its
1-skeleton, and $\otimes_{i=1}^m X_i$ has the same Cartesian
automorphisms as $\otimes_{i=1}^m X_i^1$. Hence a vertex stabilizer
$G_v$ of a component $X$ of $\otimes_{i=1}^m X_i$ has $2^m\cdot m!$
Cartesian automorphisms, and $|G_v|$ is at most the cardinality of
the stabilizer of $v$ in $X^1$, which is $2^m\cdot m!$ by
Proposition \ref{h2}.
\end{proof}

\begin{rem}
We do need the condition $n\ge 3$ in Proposition \ref{h2} and
Corollary \ref{h3}. For $n=2$, Figure \ref{figh1} illustrates a
component of the tensor product of two squares. Its 1-skeleton is
the complete bipartite graph $K_{4,4}$ with lots of non-Cartesian
automorphisms. With the face structure, there are much fewer complex
automorphisms, but swapping $(0,2)$ and $(2,0)$ still gives a
non-Cartesian complex automorphism.
\end{rem}

Now we formally define the face blocks mentioned in the beginning of
the chapter. An intuitive definition of a face block in a complex
tensor product $\otimes_{i=1}^m X_i$ would be any connected
component in $\otimes_{i=1}^m f_i$, where each $f_i$ is a face of
$X_i$. Note that if each $f_i$ is an even gon attached injectively,
then $\otimes_{i=1}^m f_i$ has $2^{m-1}$ components, and hence
$2^{m-1}$ face blocks. If these $f_i$'s are attached
non-injectively, then the above face blocks could have extra
incidence relations, and we might end up having fewer components. We
would like to define a face block regardless of attaching maps, so
we take the following definition.

\begin{defn}\label{h4}
For $i\in\{1,2,\ldots,m\}$, let $X_i$ be a polygonal cell complex
with only faces of even length $2n\ge 2$. Let $f_i$ be a face of
$X_i$ with corners labeled by $0,1,\ldots,2n-1$ cyclically. A {\bf
face block} generated by $f_1,f_2,\ldots,f_m$ is a subcollection of
faces generated by $f_1,f_2,\ldots,f_m$ such that two faces $f_a$
and $f_b$ are in the same face block if and only if a corner of
$f_a$ with label $(a_1,a_2,\ldots,a_m)$ and a corner $f_b$ with
label $(b_1,b_2,\ldots,b_m)$ have $$a_1-b_1\equiv
a_2-b_2\equiv\cdots\equiv a_m-b_m\hspace{0mm}\mod 2.$$
\end{defn}

\begin{rem} It is easy to see that a face block is well-defined no matter how
faces are cyclically labeled and no matter which corners are chosen
to verify the above criterion. In general it is not obvious whether
or not two faces are in the same face block of a complex tensor
product without knowing the tensor product structure. In the tensor
product of the following class of complexes, recognizing a face
block is much easier.
\end{rem}

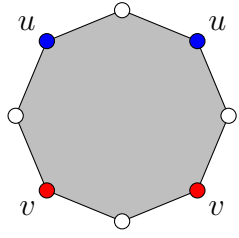
\begin{figure}
\begin{center}
\tikzstyle{place}=[circle,draw=black,fill=white,inner
sep=0pt,minimum size=2mm]
\begin{tikzpicture}[scale=1.4]

\draw [fill=lightgray]
(0:1)--(45:1)--(90:1)--(135:1)--(180:1)--(225:1)--(270:1)--(315:1)--cycle;

\node at (0:1) [place] {};

\node at (45:1) [place,fill=blue] {};

\node at (45:1) [above right] {$u$};

\node at (90:1) [place] {};

\node at (135:1) [place,fill=blue] {};

\node at (135:1) [above left] {$u$};

\node at (180:1) [place] {};

\node at (225:1) [place,fill=red] {};

\node at (225:1) [below left] {$v$};

\node at (270:1) [place] {};

\node at (315:1) [place,fill=red] {};

\node at (315:1) [below right] {$v$};

\end{tikzpicture}
\caption{a non-elementary complex}\label{figh2}
\end{center}
\end{figure}

\begin{defn}\label{h5} A connected polygonal cell complex $X$ is an {\bf
elementary} complex if $X$ satisfies the following three conditions:

\vspace{-3pt}
\begin{itemize}
\itemsep=-3pt
\item  [(1)] Every face of $X$ is of the same even length $\ge 2$.
\item  [(2)] No antipodal corners of a face are attached to the same vertex.
\item  [(3)] For any two vertices, there is at most one pair of
antipodal face corners attached.
\end{itemize}
\vspace{-8pt}
\end{defn}
\begin{rem}
Condition (3) basically says no two faces can be attached
antipodally, and in a face different pairs of antipodal corners are
not attached to the same pair of vertices. For example, the complex
in Figure \ref{figh2} is not an elementary complex.
\end{rem}

\begin{prop}\label{h6}
For $i\in\{1,2,\ldots,m\}$, let $X_i$ be an elementary complex with
faces of even length $2n\ge2$. Then in the complex tensor product
$\otimes_{i=1}^m X_i$, for any antipodal vertices $u$ and $v$ of a
face in $\otimes_{i=1}^m X_i$, there are exactly $2^{m-1}$ faces
having $u$ and $v$ as antipodal vertices, and these faces are in the
same face block. Moreover, for any two faces $f$ and $f'$ in the
same face block, we can find a series of faces $f_0,f_1,\ldots,f_k$
such that $f_0=f$, $f_k=f'$, $f_i$ and $f_{i+1}$ share antipodal
vertices for $i\in\{0,1,\ldots,k-1\}$, and $k\le n$.
\end{prop}
\begin{proof}

In $\otimes_{i=1}^m X_i$, suppose that $u=(u_1,u_2,\ldots,u_m)$ and
$v=(v_1,v_2,\ldots,v_m)$ are antipodal vertices of a face $f$
generated by $f_1,f_2,\ldots,f_m$, where $u_i$ and $v_i$ are
vertices of $X_i$ and $f_i$ is a face of $X_i$ for
$i\in\{1,2,\ldots,m\}$. Note that for each $i\in\{1,2,\ldots,m\}$,
projecting $f$ to $X_i$ gives $f_i$, and $f_i$ has $u_i$ and $v_i$
as antipodal vertices. Since $X_i$ is elementary, $u_i$ and $v_i$
are not the same vertex, and $f_i$ is the only face of $X_i$ having
$u_i$ and $v_i$ as antipodal vertices, with a unique pair of
antipodal corners attached to $u_i$ and $v_i$. Hence any face in
$\otimes_{i=1}^m X_i$ having $u$ and $v$ as antipodal vertices must
be generated by $f_1,f_2,\ldots,f_m$ in such a way that the
corresponding corners $c_i$ of the $f_i$'s at $u_i$ are combined
together. With the corner $c_1$ of $f_1$ fixed, flipping $f_i$ at
$c_i$ for $i\in\{1,2,\ldots,m\}$ gives all $2^{m-1}$ faces having
$u$ and $v$ as antipodal vertices, and these faces are in the same
face block.

Now suppose that $f$ and $f'$ are two faces in the same face block
$B$ generated by faces with corners labeled by $0,1,\ldots,2n-1$
cyclically. Then we can label corners in $B$ according to such a
corner labeling, and by following steps of
$(\pm1,\pm1,\ldots,\pm1)$, we can start from a vertex $v$ of $f$ to
reach any other vertex in $B$ in $n$ steps. In particular, there is
a unique vertex in $B$ such that we need $n$ steps to reach it from
$v$. Since $f'$ has more than one vertex, we can start from $v$ to
reach a vertex $u$ of $f'$ in $n-1$ steps. By adding one step in $f$
and one step in $f'$ if necessary, we can find a path from $f$ to
$f'$ of length at most $n+1$ such that the first and the last steps
are in $f$ and $f'$ respectively. Note that each
$(\pm1,\pm1,\ldots,\pm1)$ step determines a unique face in $B$, and
hence the above path determines a series of faces
$f_0,f_1,\ldots,f_k$ such that $f_0=f$, $f_k=f'$, and $k\le n$. If
$f_i$ and $f_{i+1}$ are determined by the same
$(\pm1,\pm1,\ldots,\pm1)$ step, then $f_i$ and $f_{i+1}$ are
actually the same face, and we can remove one of them from the
sequence. If $f_i$ and $f_{i+1}$ are determined by different
$(\pm1,\pm1,\ldots,\pm1)$ steps, then $f_i$ and $f_{i+1}$ are two
different faces with a common vertex with label
$(a_1,a_2,\ldots,a_m)$. Note that
$$(a_1,a_2,\ldots,a_m)+n(\pm1,\pm1,\ldots,\pm1)=(a_1+n,a_2+n,\ldots,a_m+n) \mod 2n,$$
which is also a common vertex of $f_i$ and $f_{i+1}$, and therefore
$f_i$ and $f_{i+1}$ share antipodal vertices. The above argument is
illustrated in Figure \ref{figh1}.
\end{proof}

Proposition \ref{h6} allows us to easily recognize a face block in a
complex tensor product. If we impose the following conditions on
each factor, then we can read the Cartesian structure of a complex
tensor product through the incidence relation of face blocks.

\begin{defn}\label{h7}

A connected polygonal cell complex $X$ is an {\bf ordinary} complex
if every face $f$ of $X$ is of the same even length $2n\ge 4$, and
satisfies the following extra conditions:

\vspace{-3pt}
\begin{itemize}
\itemsep=-3pt
\item  [(1)] If we label corners of $f$ cyclically from $1$ to $2n$,
then any two corners with different parities are not attached to the
same vertex.

\item  [(2)] For any face $f'$ incident to $f$, either $f$ has only one corner meeting $f'$, or $f$ has only two consecutive corners meeting
$f'$.

\end{itemize}
\vspace{-14pt}
\end{defn}
\begin{rem}
If the 1-skeleton of $X$ is bipartite, then $X$ satisfies (1)
automatically. Also note that a polygonal complex satisfies both (1)
and (2). The reader might have noticed that (2) implies the
condition (3) of an elementary complex. Since there are alternative
conditions serving our purpose as effectively as (2), we avoid
defining ordinary complexes as a subclass of elementary complexes.
\end{rem}

\begin{prop}\label{h8}
For $i\in\{1,2,\ldots,m\}$, suppose that $X_i$ is an ordinary
complex with faces of even length $2n\ge4$. Let $B$ be a face block
generated by $f_1,f_2,\ldots,f_m$ and $B'$ be a face block generated
by $f_1',f_2',\ldots,f_m'$, where $f_i$ and $f_i'$ are faces of
$X_i$. If $B$ and $B'$ are incident, then the following two
statements are equivalent:

\vspace{-3pt}
\begin{itemize}
\itemsep=-3pt
\item  [(1)] $\exists j$ such that $f_j$ is incident to $f_j'$ in
$X_j$, and $\forall i\neq j$ we have $f_i=f_i'$.
\item  [(2)] Every face of $B$ is incident to a face of $B'$.
\end{itemize}
\vspace{-8pt}
\end{prop}
\begin{proof}
Assume (1). Without loss of generality, we can assume that $j=1$.
Since $B$ and $B'$ are incident, there is a face corner $c$ of $B$
meeting a face corner $c'$ of $B'$. Suppose that $c$ is the
combination of corners $c_i$ of the $f_i$'s, and $c'$ is the
combination of corners $c_i'$ of the $f_i'$'s. Note that $c_1$ of
$f_1$ meets $c_1'$ of $f_1'$ in $X_1$. Also note that for $i\neq 1$,
$c_i$ and $c_i'$ are in the same face $f$, and they are either the
same corner or different corners attached to the same vertex. In
particular, by condition (1) of Definition \ref{h7}, $c_i$ and
$c_i'$ have the same parity under cyclic $\mathbb{Z}_{2n}$ labeling
for $i\neq 1$. Let $f$ be an arbitrary face of $B$ generated by
combining $c_1$ of $f_1$ with corners $\overline{c_i}$ of the
$f_i$'s for $i\neq 1$. By Definition \ref{h4}, $\overline{c_i}$ has
the same parity as $c_i$, and therefore has the same parity as
$c_i'$. Then again by Definition \ref{h4}, the face $f'$ generated
by combining $c_1'$ of $f_1'$ with $\overline{c_i}$'s of the $f_i$'s
is a face of $B'$. It is obvious that $f$ is incident to $f'$. To
summarize, given an arbitrary face $f$ of $B$, we can find a face
$f'$ of $B'$ incident to $f$. Hence (1) implies (2).

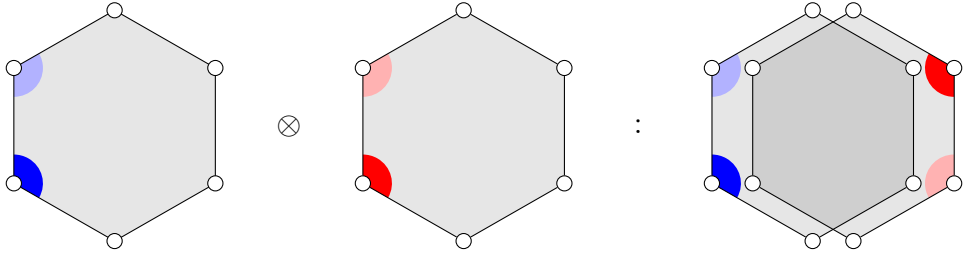
\begin{figure}
\begin{center}
\tikzstyle{place}=[circle,draw=black,fill=white,inner
sep=0pt,minimum size=2mm]
\begin{tikzpicture}[scale=1.53]

\fill [opacity=.1]
(30:1)--++(150:1)--++(210:1)--++(270:1)--++(330:1)--++(30:1)--cycle;

\fill [opacity=.1]
(3,0)++(30:1)--++(150:1)--++(210:1)--++(270:1)--++(330:1)--++(30:1)--cycle;

\fill [opacity=.1]
(6,0)++(30:1)--++(150:1)--++(210:1)--++(270:1)--++(330:1)--++(30:1)--cycle;

\fill [opacity=.1]
(6.35,0)++(30:1)--++(150:1)--++(210:1)--++(270:1)--++(330:1)--++(30:1)--cycle;

\fill [blue!30] (150:1)--+(270:.25) arc (270:390:.25)--cycle;

\fill [blue] (210:1)--+(-30:.25) arc (-30:90:.25)--cycle;

\fill [blue!30] (6,0)++(150:1)--+(270:.25) arc (270:390:.25)--cycle;

\fill [blue] (6,0)++(210:1)--+(-30:.25) arc (-30:90:.25)--cycle;

\fill [red!30] (3,0)++(150:1)--+(270:.25) arc (270:390:.25)--cycle;

\fill [red] (3,0)++(210:1)--+(-30:.25) arc (-30:90:.25)--cycle;

\fill [red!30] (6.35,0)++(-30:1)--+(90:.25) arc (90:210:.25)--cycle;

\fill [red] (6.35,0)++(30:1)--+(150:.25) arc (150:270:.25)--cycle;

\draw
(30:1)--++(150:1)--++(210:1)--++(270:1)--++(330:1)--++(30:1)--cycle;

\node at (30:1) [place] {};

\node at (90:1) [place] {};

\node at (150:1) [place] {};

\node at (210:1) [place] {};

\node at (270:1) [place] {};

\node at (330:1) [place] {};

\node at (1.5,0) {$\otimes$};

\node at (4.5,0) {:};

\node (o) at (3,0) {};

\path (o) +(30:1) node (a) [place] {} +(90:1) node (b) [place] {}
+(150:1) node (c) [place] {} +(210:1) node (d) [place] {} +(270:1)
node (e) [place] {} +(330:1) node (f) [place] {};

\draw (a)--(b)--(c)--(d)--(e)--(f)--(a);

\node (x) at (6,0) {};

\path (x) +(30:1) node (g) [place] {} +(90:1) node (h) [place] {}
+(150:1) node (i) [place] {} +(210:1) node (j) [place] {} +(270:1)
node (k) [place] {} +(330:1) node (l) [place] {};

\draw (g)--(h)--(i)--(j)--(k)--(l)--(g);

\node (y) at (6.35,0) {};

\path (y) +(30:1) node (m) [place] {} +(90:1) node (n) [place] {}
+(150:1) node (o) [place] {} +(210:1) node (p) [place] {} +(270:1)
node (q) [place] {} +(330:1) node (r) [place] {};

\draw (m)--(n)--(o)--(p)--(q)--(r)--(m);

\end{tikzpicture}
\caption{how to avoid incidence}\label{figh3}
\end{center}
\end{figure}

Assume (2). If $f_i$ and $f_i'$ are disjoint, then $B$ and $B'$ are
disjoint, which contradicts (2). Hence for each
$i\in\{1,2,\ldots,m\}$, $f_i$ and $f_i'$ are either incident or
actually the same. Suppose that there is more than one $j$, say for
$j\in\{1,2\}$, such that $f_j$ and $f_j'$ are incident. By condition
(2) of Definition \ref{h7}, $f_1$ and $f_2$ have either one corner
or two consecutive corners meeting $f_1'$ and $f_2'$ respectively.
Pick two consecutive corners of $f_1$ containing all corners meeting
$f_1'$ and colour them blue. Similarly pick two consecutive corners
of $f_2$ containing all corners meeting $f_2'$ and colour them red.
Consider the faces generated by $f_1,f_2,\ldots,f_m$ with the
following corner combination: coloured corners of $f_1$ and $f_2$
are placed at the opposite positions, as illustrated in Figure
\ref{figh3}. Note that these faces are disjoint with faces generated
by $f_1',f_2',\ldots,f_m'$. If $B$ does not contain any of these
faces, we can flip two red corners of $f_2$ to generate faces of
$B$, and the resulting faces are still disjoint with faces generated
by $f_1',f_2',\ldots,f_m'$. In other words, we can find a face of
$B$ incident to no face in $B'$, a contradiction. So there is at
most one $j$ such that $f_j$ and $f_j'$ are incident. Moreover,
condition (1) of Definition \ref{h7} implies that different face
blocks generated by $f_1,f_2,\ldots,f_m$ are disjoint. Since $B$ and
$B'$ are incident, we know that there is exactly one $j$ such that
$f_j$ and $f_j'$ are incident. Hence (2) implies (1).
\end{proof}
\begin{rem}
Note that condition (2) of Definition \ref{h7} is only used for the
argument illustrated in Figure \ref{figh3}. It is not hard to have
alternative conditions serving this purpose, especially when the
length of faces is higher. We also want to point out that through
finer examination of incidence relation between face blocks, it is
possible to obtain more information such as how $f_j$ meets $f_j'$
in $X_j$, perhaps under weaker conditions.
\end{rem}

With Propositions \ref{h6} and \ref{h8}, in a tensor product
$X=X_1\otimes X_2\otimes\cdots\otimes X_m$ where each $X_i$ is an
elementary ordinary complex with only faces of even length $2n\ge
4$, we can recognize face blocks and the Cartesian structure of $X$
through the incidence relation on faces, which is preserved under
automorphisms of $X$. Now we define a graph $\Gamma_X$ to encode the
Cartesian structure of $X$. Let $\Gamma_X$ be a simple graph with
vertex set $\times_{i=1}^m F(X_i)$, where a vertex
$(f_1,f_2,\ldots,f_m)$ represents all faces of $X$ generated by
$f_1,f_2,\ldots,f_m$, such that two vertices are adjacent if and
only if they take the same face in $m-1$ coordinates, and have
incident faces in the remaining coordinate. Let $\Gamma_{X_i}$ be a
simple graph with vertex set $F(X_i)$, such that two vertices are
adjacent if and only if the corresponding faces are incident in
$X_i$. Notice that
$\Gamma_X=\Gamma_{X_1}\Box\,\Gamma_{X_2}\Box\cdots\Box\,
\Gamma_{X_i}$, where $\Box$ is the Cartesian product of graphs (see
\cite{product} for the definition). Figure \ref{figh4} illustrates
the case $m=2$, where $B^{i,j}=(f_1^i,f_2^j)$ represents all faces
generated by $f_1^i$ and $f_2^j$. The following theorem due to
Imrich \cite{imrich} and Miller \cite{miller} restricts the
automorphism group of $\Gamma_X$.

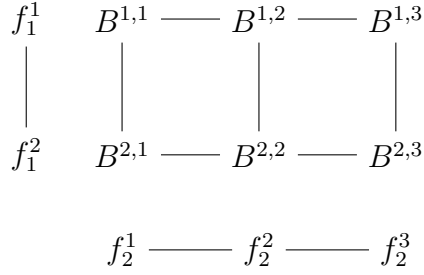
\begin{figure}
\begin{center}
\tikzstyle{place}=[circle,draw=black,fill=white,inner
sep=0pt,minimum size=2mm]
\begin{tikzpicture}[scale=1.8]

\node (a) at (0,.3) {};

\path (a) ++(.7,0) node (1) {$f_2^1$} ++(1,0) node (2) {$f_2^2$}
++(1,0) node (3) {$f_2^3$};

\node (c) at (0,1) {$f_1^2$};

\path (c) ++(.7,0) node (7) {$B^{2,1}$} ++(1,0) node (8) {$B^{2,2}$}
++(1,0) node (9) {$B^{2,3}$};

\node (b) at (0,2) {$f_1^1$};

\path (b) ++(.7,0) node (4) {$B^{1,1}$} ++(1,0) node (5) {$B^{1,2}$}
++(1,0) node (6) {$B^{1,3}$};

\draw (4)--(5)--(6) (7)--(8)--(9) (7)--(4) (8)--(5) (9)--(6) ;

\draw (1)--(2)--(3) (b)--(c);

\end{tikzpicture}
\caption{Cartesian structure of face blocks}\label{figh4}
\end{center}
\end{figure}

\begin{thm}\label{h9}
Suppose that $\Gamma$ is a finite simple connected graph with a
factorization
$\Gamma=\Gamma_1\Box\,\Gamma_2\Box\cdots\Box\,\Gamma_m$, where each
$\Gamma_i$ is prime with respect to Cartesian product. Then the
automorphism group of $\Gamma$ is generated by automorphisms of
prime factors and permutations of isomorphic factors.
\end{thm}

We can not guarantee $\Gamma_{X_i}$ is prime, but at least $X_i$ is
indeed a prime complex.

\begin{prop}\label{h10}
Let $Y$ be an elementary complex. Then $Y$ is a prime with respect
to complex tensor product, and $Y$ is not a component of any complex
tensor product.
\end{prop}
\begin{proof}
Suppose that there exist complexes $Y_1$ and $Y_2$ such that $Y$ is
a component of $Y_1\otimes Y_2$. Note that a face of $Y$ is of even
length, and must be generated by either two even faces or by one
even and one odd face. In either case, by Definition \ref{e5}, $Y$
will have faces antipodally attached together, violating that $Y$ is
elementary.
\end{proof}

\begin{figure}
\begin{center}
\tikzstyle{place}=[circle,draw=black,inner sep=0pt,minimum size=2mm]
\begin{tikzpicture}[scale=1]

\node at (-.35,-.3) {$v$}; \node at (10.6,-.3) {$v$};

\node at (-.45,2.15) {$v_5$}; \node at (10.68,2.15) {$v_2$};

\node at (-.4,2.5) {$v_3$}; \node at (10.73,2.5) {$v_6$};

\node at (-.35,2.85) {$v_1$}; \node at (10.78,2.85) {$v_4$};

\node at (-.45,4.95) {$v_2$}; \node at (10.68,4.95) {$v_5$};

\node at (-.4,5.3) {$v_6$}; \node at (10.73,5.3) {$v_3$};

\node at (-.35,5.65) {$v_4$}; \node at (10.78,5.65) {$v_1$};

\path [fill,opacity=.1] (0,-.3)-- ++(45:1)-- ++(2,0)-- ++(-45:1)--
++(45:1)-- ++(2,0)-- ++(-45:1)-- ++(45:1)-- ++(2,0)-- ++(-45:1)--
++(225:1)-- ++(-2,0)-- ++(135:1)-- ++(225:1)-- ++(-2,0)--
++(135:1)-- ++(225:1)-- ++(-2,0)-- ++(135:1);

\path [fill,opacity=.1] (0,2.2)-- ++(45:1)-- ++(2,0)-- ++(-45:1)--
++(45:1)-- ++(2,0)-- ++(-45:1)-- ++(45:1)-- ++(2,0)-- ++(-45:1)--
++(225:1)-- ++(-2,0)-- ++(135:1)-- ++(225:1)-- ++(-2,0)--
++(135:1)-- ++(225:1)-- ++(-2,0)-- ++(135:1);

\path [fill,opacity=.1] (0.05,2.5)-- ++(45:1)-- ++(2,0)--
++(-45:1)-- ++(45:1)-- ++(2,0)-- ++(-45:1)-- ++(45:1)-- ++(2,0)--
++(-45:1)-- ++(225:1)-- ++(-2,0)-- ++(135:1)-- ++(225:1)--
++(-2,0)-- ++(135:1)-- ++(225:1)-- ++(-2,0)-- ++(135:1);

\path [fill,opacity=.1] (0.1,2.8)-- ++(45:1)-- ++(2,0)-- ++(-45:1)--
++(45:1)-- ++(2,0)-- ++(-45:1)-- ++(45:1)-- ++(2,0)-- ++(-45:1)--
++(225:1)-- ++(-2,0)-- ++(135:1)-- ++(225:1)-- ++(-2,0)--
++(135:1)-- ++(225:1)-- ++(-2,0)-- ++(135:1);

\path [fill,opacity=.1] (0,5)-- ++(45:1)-- ++(2,0)-- ++(-45:1)--
++(45:1)-- ++(2,0)-- ++(-45:1)-- ++(45:1)-- ++(2,0)-- ++(-45:1)--
++(225:1)-- ++(-2,0)-- ++(135:1)-- ++(225:1)-- ++(-2,0)--
++(135:1)-- ++(225:1)-- ++(-2,0)-- ++(135:1);

\path [fill,opacity=.1] (0.05,5.3)-- ++(45:1)-- ++(2,0)--
++(-45:1)-- ++(45:1)-- ++(2,0)-- ++(-45:1)-- ++(45:1)-- ++(2,0)--
++(-45:1)-- ++(225:1)-- ++(-2,0)-- ++(135:1)-- ++(225:1)--
++(-2,0)-- ++(135:1)-- ++(225:1)-- ++(-2,0)-- ++(135:1);

\path [fill,opacity=.1] (0.1,5.6)-- ++(45:1)-- ++(2,0)-- ++(-45:1)--
++(45:1)-- ++(2,0)-- ++(-45:1)-- ++(45:1)-- ++(2,0)-- ++(-45:1)--
++(225:1)-- ++(-2,0)-- ++(135:1)-- ++(225:1)-- ++(-2,0)--
++(135:1)-- ++(225:1)-- ++(-2,0)-- ++(135:1);

\path [fill,opacity=.1] (-2.9,2.2) -- ++(45:1)-- ++(0,2)--
++(135:1)-- ++(225:1)-- ++(0,-2)-- ++(-45:1);

\path (-2.9,2.2) +(0,-.1) node [below] {$u_4$} ++(45:1) node [below
right] {$u_3$}-- ++(0,2) node [above right] {$u_2$}-- ++(135:1)
+(0,.1) node [above] {$u_1$}-- ++(225:1) node [above left] {$u_6$}--
++(0,-2) node [below left] {$u_5$}-- ++(-45:1);

\draw (-2.9,2.2) node [place,fill=white] {}-- ++(45:1) node
[place,fill=orange] {}-- ++(0,2) node [place,fill=white] {}--
++(135:1) node [place,fill=orange] {}-- ++(225:1) node
[place,fill=white] {}-- ++(0,-2) node [place,fill=orange] {}--
++(-45:1) node [place,fill=white] {};

\draw (0,-.3) node [place,fill=white] {}-- ++(45:1) node
[place,fill=white] {}-- ++(2,0) node [place,fill=white] {}--
++(-45:1) node [place,fill=white] {}-- ++(45:1) node
[place,fill=white] {}-- ++(2,0) node [place,fill=white] {}--
++(-45:1) node [place,fill=white] {}-- ++(45:1) node
[place,fill=white] {}-- ++(2,0) node [place,fill=white] {}--
++(-45:1) node [place,fill=white] {}-- ++(225:1) node
[place,fill=white] {}-- ++(-2,0) node [place,fill=white] {}--
++(135:1) node [place,fill=white] {}-- ++(225:1) node
[place,fill=white] {}-- ++(-2,0) node [place,fill=white] {}--
++(135:1) node [place,fill=white] {}-- ++(225:1) node
[place,fill=white] {}-- ++(-2,0) node [place,fill=white] {}--
++(135:1) node [place,fill=white] {};

\draw (0,2.2) node [place,fill=orange] {}-- ++(45:1) node
[place,fill=white] {}-- ++(2,0) node [place,fill=orange] {}--
++(-45:1) node [place,fill=white] {}-- ++(45:1) node
[place,fill=orange] {}-- ++(2,0) node [place,fill=white] {}--
++(-45:1) node [place,fill=orange] {}-- ++(45:1) node
[place,fill=white] {}-- ++(2,0) node [place,fill=orange] {}--
++(-45:1) node [place,fill=white] {}-- ++(225:1) node
[place,fill=orange] {}-- ++(-2,0) node [place,fill=white] {}--
++(135:1) node [place,fill=orange] {}-- ++(225:1) node
[place,fill=white] {}-- ++(-2,0) node [place,fill=orange] {}--
++(135:1) node [place,fill=white] {}-- ++(225:1) node
[place,fill=orange] {}-- ++(-2,0) node [place,fill=white] {}--
++(135:1) node [place,fill=orange] {};

\draw (0.05,2.5) node [place,fill=orange] {}-- ++(45:1) node
[place,fill=white] {}-- ++(2,0) node [place,fill=orange] {}--
++(-45:1) node [place,fill=white] {}-- ++(45:1) node
[place,fill=orange] {}-- ++(2,0) node [place,fill=white] {}--
++(-45:1) node [place,fill=orange] {}-- ++(45:1) node
[place,fill=white] {}-- ++(2,0) node [place,fill=orange] {}--
++(-45:1) node [place,fill=white] {}-- ++(225:1) node
[place,fill=orange] {}-- ++(-2,0) node [place,fill=white] {}--
++(135:1) node [place,fill=orange] {}-- ++(225:1) node
[place,fill=white] {}-- ++(-2,0) node [place,fill=orange] {}--
++(135:1) node [place,fill=white] {}-- ++(225:1) node
[place,fill=orange] {}-- ++(-2,0) node [place,fill=white] {}--
++(135:1) node [place,fill=orange] {};

\draw (0.1,2.8) node [place,fill=orange] {}-- ++(45:1) node
[place,fill=white] {}-- ++(2,0) node [place,fill=orange] {}--
++(-45:1) node [place,fill=white] {}-- ++(45:1) node
[place,fill=orange] {}-- ++(2,0) node [place,fill=white] {}--
++(-45:1) node [place,fill=orange] {}-- ++(45:1) node
[place,fill=white] {}-- ++(2,0) node [place,fill=orange] {}--
++(-45:1) node [place,fill=white] {}-- ++(225:1) node
[place,fill=orange] {}-- ++(-2,0) node [place,fill=white] {}--
++(135:1) node [place,fill=orange] {}-- ++(225:1) node
[place,fill=white] {}-- ++(-2,0) node [place,fill=orange] {}--
++(135:1) node [place,fill=white] {}-- ++(225:1) node
[place,fill=orange] {}-- ++(-2,0) node [place,fill=white] {}--
++(135:1) node [place,fill=orange] {};

\draw (0,5) node [place,fill=white] {}-- ++(45:1) node
[place,fill=orange] {}-- ++(2,0) node [place,fill=white] {}--
++(-45:1) node [place,fill=orange] {}-- ++(45:1) node
[place,fill=white] {}-- ++(2,0) node [place,fill=orange] {}--
++(-45:1) node [place,fill=white] {}-- ++(45:1) node
[place,fill=orange] {}-- ++(2,0) node [place,fill=white] {}--
++(-45:1) node [place,fill=orange] {}-- ++(225:1) node
[place,fill=white] {}-- ++(-2,0) node [place,fill=orange] {}--
++(135:1) node [place,fill=white] {}-- ++(225:1) node
[place,fill=orange] {}-- ++(-2,0) node [place,fill=white] {}--
++(135:1) node [place,fill=orange] {}-- ++(225:1) node
[place,fill=white] {}-- ++(-2,0) node [place,fill=orange] {}--
++(135:1) node [place,fill=white] {};

\draw (0.05,5.3) node [place,fill=white] {}-- ++(45:1) node
[place,fill=orange] {}-- ++(2,0) node [place,fill=white] {}--
++(-45:1) node [place,fill=orange] {}-- ++(45:1) node
[place,fill=white] {}-- ++(2,0) node [place,fill=orange] {}--
++(-45:1) node [place,fill=white] {}-- ++(45:1) node
[place,fill=orange] {}-- ++(2,0) node [place,fill=white] {}--
++(-45:1) node [place,fill=orange] {}-- ++(225:1) node
[place,fill=white] {}-- ++(-2,0) node [place,fill=orange] {}--
++(135:1) node [place,fill=white] {}-- ++(225:1) node
[place,fill=orange] {}-- ++(-2,0) node [place,fill=white] {}--
++(135:1) node [place,fill=orange] {}-- ++(225:1) node
[place,fill=white] {}-- ++(-2,0) node [place,fill=orange] {}--
++(135:1) node [place,fill=white] {};

\draw (0.1,5.6) node [place,fill=white] {}-- ++(45:1) node
[place,fill=orange] {}-- ++(2,0) node [place,fill=white] {}--
++(-45:1) node [place,fill=orange] {}-- ++(45:1) node
[place,fill=white] {}-- ++(2,0) node [place,fill=orange] {}--
++(-45:1) node [place,fill=white] {}-- ++(45:1) node
[place,fill=orange] {}-- ++(2,0) node [place,fill=white] {}--
++(-45:1) node [place,fill=orange] {}-- ++(225:1) node
[place,fill=white] {}-- ++(-2,0) node [place,fill=orange] {}--
++(135:1) node [place,fill=white] {}-- ++(225:1) node
[place,fill=orange] {}-- ++(-2,0) node [place,fill=white] {}--
++(135:1) node [place,fill=orange] {}-- ++(225:1) node
[place,fill=white] {}-- ++(-2,0) node [place,fill=orange] {}--
++(135:1) node [place,fill=white] {};

\node at (1.707,-.3) {$f^1$}; \node at (5.121,-.3) {$f^2$}; \node at
(8.535,-.3) {$f^3$}; \node at (-2.9,3.907) {$f$};

\end{tikzpicture}
\caption{tensor product of a hexagon with a 3-hexagon
necklace}\label{figh5}
\end{center}
\end{figure}
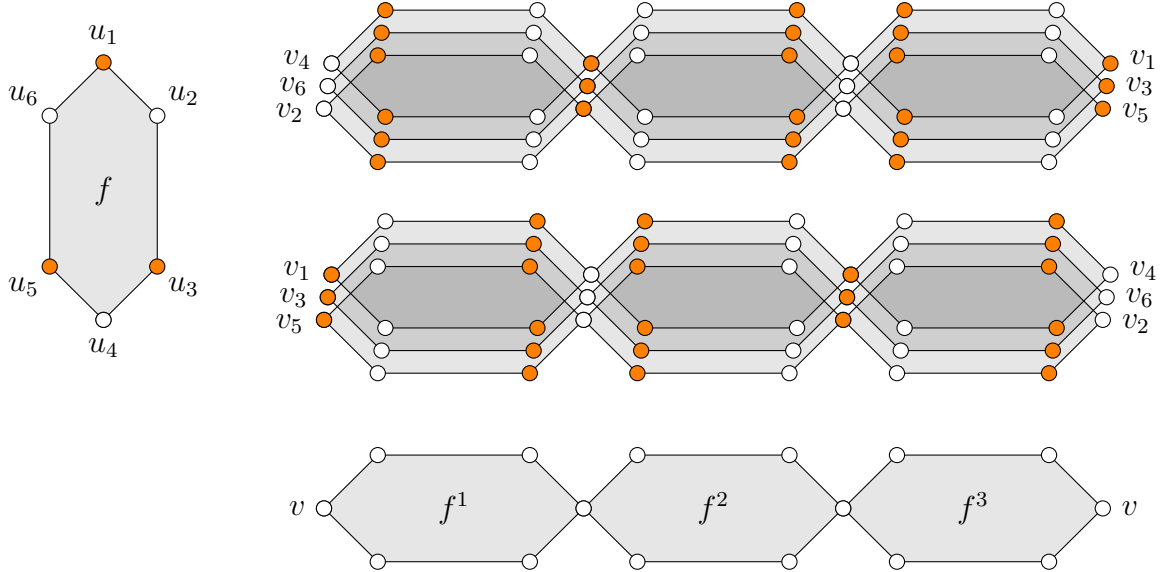

Note that in Figure \ref{figh4}, each $B^{i,j}$ actually contains
two face blocks generated by $f_1^i$ and $f_2^j$, and in general
each vertex of $\Gamma_X$ defined above contains $2^{m-1}$ face
blocks. Even if we have some control over the automorphism group of
$\Gamma_X$, having multiple face blocks at one vertex of $\Gamma_X$
could lead to non-Cartesian automorphisms of $X$. Let us look at the
tensor product of a hexagon with a 3-hexagon necklace as illustrated
in Figure \ref{figh5}, where $v_i$ is the vertex generated by $u_i$
and $v$, and coloured vertices in the product are generated by
coloured $u_1$, $u_3$, and $u_5$. For brevity, half of the faces in
the product are omitted. Consider the automorphism $\rho$ of the
product induced by fixing $f$, $f^1$, and $f^3$ but flipping $f^2$
(swapping the top and the bottom edges) in two factors. Then $\rho$
fixes the four face blocks on the left and right, and permutes
vertices in each of the two middle blocks. In particular, we can
permute vertices in a block while its two incident blocks are fixed.
Therefore we can permute vertices in one middle block and fix all
other five blocks. This gives a non-Cartesian automorphism.

There are two main reasons why we have the above non-Cartesian
automorphism. First, there is more than one face block generated by
the same faces lying in the same component of the product. Secondly,
factors are not rigid enough, so the action on one face block can
not affect incident blocks, and can not be transmitted to blocks
generated by the same faces. We suspect that if either of these two
reasons is absent, then each component of the product might have
only Cartesian automorphisms. In particular, if the 1-skeleton of
each factor is bipartite, then face blocks generated by the same
faces are in different components. Also note that if a complex is a
surface, it is rigid enough that the action on one face completely
determines the whole automorphism. So far we do not have a definite
result yet, and hence we pose the following two conjectures. We hope
to resolve these problems in the near future.

\begin{conj}\label{h11}
For $i\in\{1,2,\ldots,m\}$, suppose that $X_i$ is an elementary
ordinary complex with faces of the same even length $2n\ge 6$, and
$X_i$ has bipartite 1-skeleton. Then for any component $X$ of the
complex tensor product $\otimes_{i=1}^m X_i$, $\Aut(X)$ can be
generated by automorphisms of $X_i$'s together with permutations of
isomorphic factors.
\end{conj}

\begin{conj}\label{h12}
For $i\in\{1,2,\ldots,m\}$, suppose that $X_i$ is an elementary
ordinary complex with faces of the same even length $2n\ge 6$, and
$X_i$ has surface structure. Then for any component $X$ of the
complex tensor product $\otimes_{i=1}^m X_i$, $\Aut(X)$ can be
generated by automorphisms of $X_i$'s together with permutations of
isomorphic factors.
\end{conj}

\section{Acknowledgements}

The author appreciates the support of National Center for
Theoretical Science, Taiwan, and would like to thank Ian Leary for
introducing this topic and valuable advices.

\end{document}